\documentclass[runningheads]{llncs}
\usepackage{amsmath, amssymb, amsfonts}
\newtheorem{fact}[theorem]{Fact}
\newtheorem{notation}[theorem]{Notation}

\newcommand{\COMMENTED}[1]{}
\usepackage{xcolor}

\newcommand{\mycite}[2]{{\rm\cite[\textsc{#1}]{#2}}}

\newcommand{\IQ}{\mathbb{Q}}
\newcommand{\IN}{\mathbb{N}}

\newcommand{\IA}{\mathbb{A}}
\newcommand{\IB}{\mathbb{B}}
\newcommand{\IP}{\mathbb{P}}
\newcommand{\IR}{\mathbb{R}}
\newcommand{\IF}{\mathbb{F}}
\newcommand{\IS}{\mathbb{S}}
\newcommand{\II}{\mathbb{I}}

\newcommand{\calS}{\mathcal{S}}

\newcommand{\dom}{\mbox{dom}}

\newcommand{\Laplace}{\bigtriangleup}
\newcommand{\Semigroup}{\mathcal{S}}
\newcommand{\onorm}[1]{{\left\vert\kern-0.25ex\left\vert\kern-0.25ex\left\vert #1
    \right\vert\kern-0.25ex\right\vert\kern-0.25ex\right\vert}}
\renewcommand{\onorm}[1]{\left|#1\right|}
\newcommand{\C}{C}
\newcommand{\Cone}{C^1}
\newcommand{\Cinfty}{C^\infty}
\newcommand{\myto}{\!\to\!}
\newcommand{\ELL}[1]{L_{#1}}
\newcommand{\SOB}[2]{H_{#1}^{#2}}
\newcommand{\deltaELL}[1]{\delta_{L_{#1}}}
\newcommand{\deltaSOLZ}[1]{\delta_{L_{#1,0}^{\sigma}}}
\newcommand{\deltaSOB}[1]{\delta_{H_{2,0}^{#1}}}
\newcommand{\SOLZ}[1]{L^{\sigma}_{#1,0}}    %with zero boundary condition
\newcommand{\SOLPZ}{\IQ^{\sigma}_{0}}       %polynomial vectors, solenoidal + zero on boundary
\newcommand{\poly}{p}
\newcommand{\p}{2}
\newcommand{\calP}{\mathcal{P}}

\newcommand{\Pressure}{P}
\newcommand{\Helmholtz}{\mathbb{P}}
\newcommand{\divergence}{\nabla\cdot}
\newcommand{\myrho}{\rho}
\newcommand{\Trim}{\mathbb{T}}

\def\qed{\unskip\nobreak\hfil
\penalty50\hskip1em\null\nobreak\hfil$\sqcup$
\parfillskip=0pt\finalhyphendemerits=0\endgraf}

\begin{document}
\title{Computability of the Solutions to Navier-Stokes Equations via Recursive Approximation\thanks{
 The third author is supported by grant NRF-2017R1E1A1A03071032.}}
\titlerunning{Computability to Navier-Stokes Equations}
% If the paper title is too long for the running head, you can set
% an abbreviated paper title here
%
\author{Shu-Ming Sun\inst{1} \and
Ning Zhong\inst{2} \and Martin Ziegler\inst{3}}
\authorrunning{S. M. Sun, N. Zhong, M. Ziegler}
% First names are abbreviated in the running head.
% If there are more than two authors, 'et al.' is used.
%
\institute{Department of Mathematics, Virginia Tech, Blacksburg, VA
24061, USA \email{sun@math.vt.edu} \and Department of Mathematical
Sciences, University of Cincinnati, Cincinnati, OH 45221, USA \\
\email{zhongn@ucmail.uc.edu} \and
School of Computing, KAIST, 291 Daehak-ro, 34141 Daejeon, Republic of Korea\\
\email{ziegler@kaist.ac.kr}}
\maketitle              % typeset the header of the contribution
\bigskip

\centerline{ {\it The paper is dedicated to the memory of Professor Ker-I Ko.}}

\begin{abstract}
As one of the seven open problems in the addendum to their 1989 book
{\it Computability in Analysis and Physics}, Pour-El
and Richards proposed ``... the recursion theoretic
study of particular nonlinear problems of classical importance.
Examples are the Navier-Stokes equation, the KdV equation, and the
complex of problems associated with Feigenbaum's constant.'' In this
paper, we approach the question of whether the Navier-Stokes
Equation admits recursive solutions in the sense of Weihrauch's
Type-2 Theory of Effectivity. A natural encoding
(``representation'') is constructed for the space of divergence-free
vector fields on 2-dimensional open square $\Omega = (-1, 1)^2$.
This representation is shown to render first the mild solution to
the Stokes Dirichlet problem and then a strong local solution to the
nonlinear inhomogeneous incompressible Navier-Stokes initial value
problem uniformly computable. Based on classical approaches, the
proofs make use of many subtle and intricate estimates which are
developed in the paper for establishing the computability results.
\\

\keywords{Navier-Stokes equations  \and Computability.}
\end{abstract}
\section{Introduction}
The (physical) Church-Turing Hypothesis \cite{Soar96} postulates
that every physical phenomenon or effect can, at least in principle,
be simulated by a sufficiently powerful digital computer up to any
desired precision. Its validity had been challenged, though, in the
sound setting of Recursive Analysis: with a computable $\Cone$
initial condition to the Wave Equation leading to an incomputable
solution \cite{PERi81,PEZh86}. The controversy was later resolved by
demonstrating that, in both physically \cite{Zieg09,BCT12} and
mathematically more appropriate Sobolev space settings, the solution
is computable uniformly in the initial data \cite{WeZh02}. Recall
that functions $f$ in a Sobolev space
are not defined pointwise but by local averages in the $L_q$ sense\footnote{%
We use $q\in[1,\infty]$ to denote the norm index, $\Pressure$ for
the pressure field, $\poly$ for polynomials, $\calP$ for a set of
trimmed and mollified tuples of the latter,
%$\frakP$ for $\calP$ combined with the Laplacians of its members,
and $\Helmholtz$ for the Helmholtz Projection.} (in particular $q=2$
corresponding to energy) with derivatives understood in the
distributional sense. This led to a series of investigations on the
computability of linear and nonlinear partial differential equations
\cite{WeZh05,WeZh06,WeZh07}.

The (incompressible) Navier-Stokes Equation
\begin{equation} \label{e:NSE}
\partial_t\vec u -\Laplace \vec u +(\vec u\cdot\nabla)\vec u
+ \nabla \Pressure = \vec f, \quad \divergence \vec u=0, \quad \vec
u(0)=\vec a, \quad \vec u\big|_{\partial\Omega}\equiv\vec 0
\end{equation}
describes the motion of a viscous incompressible fluid filling a
rigid box $\overline{\Omega}$. The vector field $\vec u=\vec u(\vec
x,t)=\big(u_{1},u_{2},\ldots,u_{d}\big)$ represents the velocity of
the fluid and $\Pressure=\Pressure(\vec x, t)$ is the scalar
pressure with gradient $\nabla P$; $\Laplace$ is the Laplace
operator; $\nabla\cdot\vec u$ denotes componentwise divergence;
$\vec u\cdot\nabla$ means, in Cartesian coordinates,
$u_{1}\partial_{x_1} + u_{2}\partial_{x_2} +\ldots
+u_{d}\partial_{x_d}$; and the function $\vec a=\vec a(\vec x)$ with
$\divergence\vec a=0$ provides the initial velocity and $\vec f$ is
a given external force. Equation~(\ref{e:NSE}) thus constitutes a
system of $d+1$ partial differential equations for $d+1$ functions.

The question of global existence and smoothness of its solutions,
even in the homogeneous case $\vec f\equiv\vec 0$, is one of the
seven Millennium Prize Problems posted by the Clay Mathematics
Institute at the beginning of the 21st century. Local existence has
been established, though, in various $L_q$ settings \cite{GiMi85};
and uniqueness of weak solutions in dimension 2, but not in
dimension 3 \mycite{\S V.1.5}{Sohr}, \mycite{\S
V.1.3.1}{BoyerFabrie}. Nevertheless, numerical solution methods have
been devised in abundance, often based on pointwise (or even
uniform, rather than $L_q$) approximation and struggling with
nonphysical artifacts \cite{PMC85}. In fact, the very last of seven
open problems listed in the addendum to \cite{PERi89} asks for a
``recursion theoretic study of $\dots$ the Navier-Stokes equation''.
Moreover it has been suggested \cite{Smit03} that hydrodynamics
could in principle be incomputable in the sense of allowing to
simulate universal Turing computation and to thus `solve' the
Halting problem. And indeed recent progress towards (a negative
answer to) the Millennium Problem \cite{Tao14} proceeds by
simulating a computational process in the vorticity dynamics to
construct a blowup in finite time for a PDE very similar to
(\ref{e:NSE}).

%%%%%%%%%%%%%%%%%%%%%%%%%%%%%%%%%%%%%%%
\subsection{Overview} \label{ss:Overview}
Using the sound framework of Recursive Analysis, we assert the
computability of a local strong solution of (\ref{e:NSE}) in the
space $\SOLZ{2}(\Omega)$ (see Section \ref{s:Represent} for
definition) %of square-integrable, solenoidal (i.e. divergence-free)
%vector fields $u:\Omega\to\IR^d$ on the open cube
%$\Omega:=(-1;1)^d$, $d=2, 3$, with zero boundary condition $\vec
%u\big|_{\partial\Omega}\equiv0$
from a given initial condition $\vec
a\in\SOLZ{2}(\Omega)$; moreover, the computation is uniform in the
initial data.
We follow a common strategy used in the classical existence proofs
\cite{Giga83a,Giga83b,GiMi85,Sohr,Tema77}:
\begin{enumerate}
\item[i)]
Eliminate the pressure $\Pressure$ by applying, to both sides of
Equation~(\ref{e:NSE}), the Helmholtz projection
$\Helmholtz:\big(\ELL{\p}(\Omega)\big)^2\to\SOLZ{\p}(\Omega)$, thus
arriving at the non-linear evolution equation
\begin{equation} \label{e:NS-E}
\partial_t \vec u + \IA\vec u + \IB\vec u = \vec g \quad (t>0), \qquad
\vec u(0)= \vec a \in \SOLZ{\p}(\Omega)
\end{equation}
where $\ELL{\p}(\Omega)$ is the set of all square-integrable
real-valued functions defined on $\Omega$, $\vec g=\Helmholtz\vec
f$, $\vec u=\Helmholtz\vec u\in\SOLZ{\p}(\Omega)$,
$\IA=-\Helmholtz\Laplace$ denotes the Stokes operator, and $\IB\vec
u=\Helmholtz\, (\vec u\cdot\nabla)\vec u$ is the nonlinearity.
\item[ii)]
Construct a mild solution $\vec v(t)\vec a = e^{-t\IA}\vec a$ of the
associated homogeneous linear equation
\begin{equation} \label{e:NS-LH}
\partial_t\vec v + \IA \vec v = \vec 0
\quad \text{for }t\geq0, \qquad \vec v(0)=\vec a\in\SOLZ{\p}(\Omega)
\end{equation}
%in the sense of a strongly continuous (and in fact analytic) semigroup
%\[ \Semigroup:\SOLZ{\p}(\Omega)\;\ni \;\vec a \;\mapsto \; (t\mapsto e^{-t\IA}\vec a)
%\;\in \; C\big([0;\infty);\SOLZ{\p}(\Omega)\big) \enspace . \]
\item[iii)]
 Rewrite (\ref{e:NS-E}) using (ii) in an integral form \cite[\S2]{GiMi85}
\begin{equation} \label{e:integral-form}
\vec u(t) = e^{-t\IA}\vec a \;+\;\int_0^t e^{-(t-s)\IA} \vec
g(s)\,ds \;-\; \int^t_0e^{-(t-s)\IA }\IB\vec u(s)\, ds \quad
\text{for }\; t\geq0
\end{equation}
and solve it by a limit/fixed-point argument using the following
iteration scheme \cite[\textsc{Eq.}~(2.1)]{GiMi85}:
\begin{equation} \label{e:iteration}
\vec v_0(t) = e^{-t\IA}\vec a + \int_0^t e^{-(t-s)\IA} \vec g(s)\,ds
, \
 \vec v_{n+1}(t) = \vec
v_0(t) - \int^t_0e^{-(t-s)\IA }\IB\vec v_n(s)\,ds
\end{equation}
\item[iv)]
 Recover the pressure $\Pressure$ from $\vec u$ by solving
\begin{equation} \label{e:gradient}
 \nabla \Pressure \;=\; \vec f-\partial_t\vec u + \Laplace \vec u - (\vec{u} \cdot \nabla ) \vec{u}
 \end{equation}
\end{enumerate}

To make use of the above strategy for deriving an algorithm to
compute the solution of (\ref{e:NSE}), there are several
difficulties which need to be dealt with. Firstly, a proper
representation is needed for coding the solenoidals. The codes
should be not only rich enough to capture the functional characters
of these vector fields but also robust enough to retain the coded
information under elementary function operations, in particular,
integration. Secondly, since the Stokes operator $\IA: \dom(\IA) \to
L^{\sigma}_{2, 0}(\Omega)$ is neither continuous nor its graph dense
in $(L^{\sigma}_{2, 0}(\Omega))^2$, there is no convenient way to
directly code $\IA$ for computing the solution of the linear
equation (\ref{e:NS-LH}). The lack of computer-accessible
information on $\IA$ makes the computation of the solution
$\vec{v}(t)\vec{a}=e^{-t\IA}\vec{a}$ of (\ref{e:NS-LH}) much more
intricate. Thirdly, since the nonlinear operator $\IB$ in the
iteration (\ref{e:iteration}) involves differentiation and
multiplication, and a mere $L^{\sigma}_{2, 0}$-code of $\vec{v}_n$
is not rich enough for carrying out these operations, it follows
that there is a need for computationally derive a stronger code for
$\vec{v}_n$ from any given $L^{\sigma}_{2, 0}$-code of $\vec{a}$ so
that $\IB \vec{v}_n$ can be computed. This indicates that the
iteration is to move back and forth among different spaces, and thus
additional care must be taken in order to keep the computations
flowing in and out without any glitches from one space to another.
To overcome those difficulties arising in the recursion theoretic
study of the Navier-Stokes equation, many estimates - subtle and
intricate - are established in addition to the classical estimates.
\\

The paper is organized as follows.  Presuming familiarity with
Weihrauch's \emph{Type-2 Theory of Effectivity} \cite{Weih00},
Section~\ref{s:Represent} recalls the standard representation
$\deltaELL{\p}$ of $\ELL{\p}(\Omega)$ and introduces a natural
representation $\deltaSOLZ{\p}$ of $\SOLZ{\p}(\Omega)$.
Section~\ref{s:Helmholtz} proves that the Helmholtz projection
$\Helmholtz:\big(\ELL{2}(\Omega)\big)^2\to\SOLZ{2}(\Omega)$ is
$\big((\deltaELL{2})^2,\deltaSOLZ{2}\big)$-computable.
Section~\ref{s:Stokes} presents the proof that the solution to the
linear homogeneous Dirichlet problem~(\ref{e:NS-LH}) is uniformly
computable from the initial condition $\vec a$.
Section~\ref{s:Nonlinear} is devoted to show that the solution to
the nonlinear Navier-Stokes problem (\ref{e:NSE}) is uniformly
computable locally.
%in the case $\p=2$:
Subsection~\ref{ss:MulDiff} recalls the Bessel (=fractional Sobolev)
space $\SOB{2}{s}(\Omega)\subseteq\ELL{2}(\Omega)$ of $s$-fold
weakly differentiable square-integrable functions on $\Omega$ and
its associated standard representation $\deltaSOB{s}$, $s\geq0$. For
$s>1$ we assert differentiation $\SOB{2}{s}(\Omega)\ni w\mapsto
\partial_x w\in\ELL{2}(\Omega)$ to be
$\big(\deltaSOB{s},\deltaELL{2}\big)$-computable and multiplication
$\SOB{2}{s}(\Omega)\times\ELL{2}(\Omega)\ni (v,w)\mapsto
vw\in\ELL{2}(\Omega)$ to be
$\big(\deltaSOB{s}\times\deltaELL{2},\deltaELL{2}\big)$-computable.
Based on these preparations, Subsection~\ref{ss:iteration} asserts
that in the homogeneous case $\vec g\equiv\vec0$,  the  sequence,
generated from the iteration map
\begin{align*}
\IS\;:\;\C\big([0;\infty),\SOLZ{2}(\Omega)\big)\times\C\big([0;\infty), & \SOLZ{2}(\Omega)\big)
\;\ni\;(\vec v_0,\vec v_n) \\
&
\;\mapsto\;\vec v_{n+1}\;\in\;\C\big([0;\infty),\SOLZ{2}(\Omega)\big)
\end{align*}
according to Equation~(\ref{e:iteration}),
converges effectively uniformly on some positive (but not necessarily
maximal) time interval $[0;T]$ whose length $T=T(\vec a)>0$ is
computable from the initial condition $\vec a$.
Subsection~\ref{ss:iteration0} proves
that the iteration map $\IS$ is
$\big([\myrho\myto\deltaSOLZ{2}]\times[\myrho\myto\deltaSOLZ{2}],[\myrho\myto\deltaSOLZ{2}]\big)$-computable.
We conclude in
Subsection~\ref{ss:final} with the final extensions regarding the
inhomogeneity $\vec f$ and pressure $\Pressure$, thus establishing
the main result of this work:

\begin{theorem} \label{t:Main}
There exists a $\displaystyle
\big(\deltaSOLZ{2}\!\!\times\!\big[\myrho\myto\deltaSOLZ{2}\big]\:,\myrho\big)
\text{-computable map } T$,
\[
T:\SOLZ{2}(\Omega)\times\C\big([0;\infty),\SOLZ{2}(\Omega)\big)\to(0;\infty),
\quad (\vec a, \vec f)\mapsto T(\vec a, \vec f) \] and a
$\big(\deltaSOLZ{2}\!\!\!\times\!\big[\myrho\!\to\!\deltaSOLZ{2}\big]\times\myrho\:,
\deltaSOLZ{2}\big)\text{-computable partial map } \cal S$,
\[
\calS:\subseteq\SOLZ{2}(\Omega)\times\C\big([0;\infty),\SOLZ{2}(\Omega)\big)\times[0;\infty)
\to\SOLZ{2}(\Omega)\!\times\!\ELL{2}(\Omega)
\]
such that, for every $\vec a\in\SOLZ{2}(\Omega)$ and $\vec
f\in\C\big([0;\infty),\SOLZ{2}(\Omega)\big)$, the function $(\vec
u,P):\big[0;T(\vec a,\vec f)]\ni t\mapsto \calS(\vec a,\vec f,t)$
constitutes a (strong local in time and weak global in space)
solution to Equation~(\ref{e:NSE}).
\end{theorem}

Roughly speaking, a function is computable if it can be approximated
by ``computer-accessible" functions (such as rational numbers,
polynomials with rational coefficients, and so forth) with arbitrary
precision, where precision is given as an input; such a sequence of
approximations is called an effective approximation. Thus in terms
of effective approximations, the theorem states that the solution of
Equation (\ref{e:NSE}) can be effectively approximated locally in
the time interval $[0, T(\vec{a}, \vec{f})]$, where the time
instance $T(\vec{a}, \vec{f})$ is effectively approximable. %

More precisely, in computable analysis, a map $F:X\to Y$ from a
space $X$ with representation $\delta_X$ to a space $Y$ with
representation $\delta_Y$ is said to be $(\delta_X,
\delta_Y)$-computable if there exists a (Turing) algorithm (or any
computer program) that computes a $\delta_Y$-name of $F(x)$ from any
given $\delta_X$-name of $x$. A metric space $(X,d)$, equipped with
a partial enumeration $\zeta:\subseteq\IN\to X$ of some dense
subset, gives rise to a canonical Cauchy representation
$\delta_\zeta$ by encoding each $x\in X$ with a sequence $\bar
s=(s_0,s_1,s_2,\ldots)\in\dom(\zeta)^\omega\subseteq\IN^\omega$ such
that $d\big(x,\zeta(s_k)\big)\leq2^{-k}$ for all $k$
\cite[\S8.1]{Weih00}; in other words, $\{ \zeta(s_k)\}$ is an
effective approximation of $x$. For example, approximating by
(dyadic) rationals thus leads to the standard representation
$\myrho$ of $\IR$; and for a fixed bounded $\Omega\subseteq\IR^2$,
the standard representation $\deltaELL{\p}$ of $\ELL{\p}(\Omega)$
encodes $f\in\ELL{\p}(\Omega)$ by a sequence
$\{\poly_k:k\in\IN\}\subseteq\IQ[\IR^2]$ of $d$-variate polynomials
with rational coefficients such that $\|f-\poly_k\|_{\p}\leq2^{-k}$,
where $\| \cdot\|_{\p}=\| \cdot \|_{L_{2}}$. Thus if both spaces $X$
and $Y$ admit Cauchy representations, then a function $f: X\to Y$ is
computable if there is an algorithm that computes an effective
approximation of $f(x)$ on any given effective approximation of $x$
as input. For represented spaces $(X,\delta_X)$ and $(Y,\delta_Y)$,
$\delta_X\!\times\!\delta_Y$ denotes the canonical representation of
the Cartesian product $X\!\times\! Y$. When $X$ and $Y$ are
$\sigma$-compact metric spaces with respective canonical Cauchy
representations $\delta_X$ and $\delta_Y$, $[\delta_X\myto\delta_Y]$
denotes a canonical representation of the space $\C(X,Y)$ of
continuous total functions $f:X\to Y$, equipped with the
compact-open topology
\mycite{Theorem~3.2.11+Definition~3.3.13}{Weih00}. The
representation $[\delta_X\myto\delta_Y]$ supports type conversion in
the following sense \mycite{Theorem~3.3.15}{Weih00}:

\begin{fact} \label{f:conversion}
On the one hand, the evaluation $(f,x)\mapsto f(x)$ is
$([\delta_X\myto\delta_Y]\!\times\!\delta_X\:,\:\delta_Y)$-computable.
On the other hand, a map $f:X\times Y\to Z$ is
$(\delta_X\!\times\!\delta_Y\:,\:\delta_Z)$-computable iff the map
$X\ni x\mapsto \big(y\mapsto f(x,y)\big)\in\C(Y,Z)$ is
$(\delta_X\:,\:[\delta_Y\myto\delta_Z])$-computable. \\
\end{fact}

We mention in passing that all spaces considered in this paper are
equipped with a norm. Thus for any space $X$ considered below, a
$\delta_{X}$-name of $f\in X$ is simply an effective approximation
of $f$ despite the often cumbersome notations. \\

%\noindent \textbf{Convention.} For readability all definitions and
%theorems will be stated and proved in the two dimensional space (and
%so $\Omega = (-1, 1)^2$).

%%%%%%%%%%%%%%%%%%%%%%%%%%%%%%%%%%%%%%%%%%%%%%%%%%%%%%%%%%
\section{Representing Divergence-Free $\ELL{\p}$ Functions on $\Omega$}
\label{s:Represent}

%In the sequel, we fix $\Omega$ to denote the $d$-dimensional open
%cube $(-1; 1)^d$, $d=2, 3$.
Let us call a vector field $\vec f$ satisfying $\divergence
\vec f=0$ in $\Omega$ %and $\vec f=0$ on $\partial\Omega$
\emph{divergence-free}. A vector-valued function $\vec\poly$ is
called a polynomial of degree $N$ if each of its components is a
polynomial of degree  less than or equal to $N$ with respect to each
variable and at least one component is a polynomial of degree $N$.
Let $\SOLZ{\p}(\Omega)$ --- or $\SOLZ{\p}$ if the context is clear
--- be the closure in $\ELL{\p}$-norm of the set $\{ \vec u\in
(\Cinfty_0(\Omega))^2 \, : \, \divergence \vec u=0\}$ of all smooth
divergence-free functions with support of $\vec u$ and all of its
partial derivatives contained in some compact subset of  $\Omega$.
Let $\IQ[\IR^2]$ be the set of all polynomials of two real variables
with rational coefficients and
%$\vec X=(X_1,\ldots,X_d)$ with rational coefficients;
$\SOLPZ[\IR^2]$ the subset of all $2$-tuples of such polynomials
which are divergence-free in $\Omega$ and vanish on $\partial
\Omega$.
We note that the boundary value of a $L^{\sigma}_{2,
0}(\Omega)$-function $\vec{u}$, $\vec{u}|_{\partial
\Omega}$,  is not defined unless $\vec{u}$ is (weakly)
differentiable; if $\vec{u}$ is (weakly) differentiable, then
$\vec{u}|_{\partial \Omega}=0$. \\

\begin{notation} \label{n:1} Hereafter we use $\| w \|_2$ for the $L_2$-norm $\|
w\|_{L_2(\Omega)}$ if $w$ is real-valued, or for $\|
w\|_{(L_2(\Omega))^2}$ if $w$ is vector-valued (in $\mathbb{R}^2$).
We note that $\| \cdot \|_{L^{\sigma}_{2, 0}(\Omega)}= \| \cdot
\|_{(L_2(\Omega))^2}$. For any subset $A$ of $\mathbb{R}^n$, its
closure is denoted as $\overline{A}$.
\end{notation}

\begin{proposition} \label{polynomial}
\begin{enumerate}
\item[a)]
A polynomial tuple
$$
\vec\poly=(\poly_1, \poly_2)=\big(\sum_{i,
j=0}^{N}a^1_{i, j}x^iy^j, \sum_{i, j=0}^{N}a^2_{i, j}x^iy^j\big)
$$
is divergence-free and boundary-free if and only if  its coefficients
satisfy the following system of linear equations with integer
coefficients:
\begin{eqnarray} \label{div}
(i+1)a^1_{i+1, j}+(j+1)a^2_{i, j+1}=0, & 0\leq i, j\leq N-1 \smallskip \nonumber \\
(i+1)a^1_{i+1, N}=0, & 0\leq i\leq N-1 \smallskip \\
(j+1)a^2_{N, j+1}=0, & 0\leq j\leq N-1 \nonumber
\end{eqnarray}
and for all $0\leq i,j\leq N$,
\begin{gather} \label{boundary-1}
\sum\nolimits_{i=0}^N a^1_{i, j}=\sum\nolimits_{i=0}^N a^2_{i, j}=\sum\nolimits_{i=0}^N
(-1)^ia^1_{i, j}=\sum\nolimits_{i=0}^N (-1)^ia^2_{i, j}=0 \\
\sum\nolimits_{j=0}^N a^1_{i, j}=\sum\nolimits_{j=0}^N a^2_{i, j}=\sum\nolimits_{j=0}^N
(-1)^ja^1_{i, j}=\sum\nolimits_{j=0}^N (-1)^ja^2_{i, j}=0  \label{boundary-2}
\end{gather}
\item[b)]
$\SOLPZ[\IR^2]$ is dense in $\SOLZ{\p}(\Omega)$ w.r.t.
$\ELL{\p}$-norm.
\end{enumerate}
\end{proposition}
The proof of Proposition \ref{polynomial} is
%elementary calculation and
deferred to Appendix \ref{a:polynomial}. \\

We may be tempted to use $\SOLPZ[\IR^2]$ as a set of names for
coding/approximating the elements in the space $\SOLZ{\p}(\Omega)$.
However, since the closure of $\SOLPZ[\IR^2]$ in $L_2$-norm contains
$\SOLZ{\p}(\Omega)$ as a proper subspace, $\SOLPZ[\IR^2]$ is ``too
big" to be used as a set of codes for representing
$\SOLZ{\p}(\Omega)$; one has to ``trim" polynomials in
$\SOLPZ[\IR^2]$ so that any convergent sequence of ``trimmed"
polynomials converges to a limit in $\SOLZ{\p}(\Omega)$. The
trimming process is shown below.
%However it
%turns out preferable to first ``trim" its elements to not only
%vanish on $\partial\Omega$ but have support in $\Omega$ (while
%remaining divergence-free and dense). For readability, we only show
%the modifying process in the case where $\Omega =(-1, 1)^2$. The
%same argument applies to the three dimensional case.
For each $k\in\IN$ (where $\IN$ is the set of all positive
integers),  let $\Omega_{k}=(-1+2^{-k};1-2^{-k})^2$. And  for each
%$f:[-1;1]^2\to\IR$, define
$\vec\poly =(\poly_1, \poly_2)\in \SOLPZ[\IR^2]$, define $\Trim_k
\vec \poly=(\Trim_k \poly_1, \Trim_k \poly_2)$, where
\begin{equation} \label{e:Trim}
\Trim_k\poly_j(x,y) \; =\; \left\{ \begin{array}{ll}
                  \poly_j(\frac{x}{1-2^{-k}}, \frac{y}{1-2^{-k}}), &
\; -1+2^{-k}\leq x, y\leq 1-2^{-k} \\
                   0, &  \mbox{otherwise}
                     \end{array} \right.
\end{equation}
$j=1, 2$. Then $\Trim_k\poly_{j}$ and $\Trim_k\vec\poly$ have the
following properties:
\begin{enumerate}
\item[a)] $\Trim_k\vec\poly$ has compact
support $\overline{\Omega}_k$ contained in $\Omega$.
\item[b)] $\Trim_k\vec\poly$ is a polynomial with rational coefficients defined in $\Omega_k$.
\item[c)] $\Trim_k\vec\poly$ is continuous on $\overline{\Omega}=[-1, 1]^2$.
\item[d)] %$G_{j, k, \vec\poly}(1-2^{-k}, y)=\poly_j(1,
%\frac{y}{1-2^{-k}})=0$ and $G_{j, k, \vec\poly}(-1+2^{-k}, y)=\poly_j(-1,
%\frac{y}{1-2^{-k}})=0$ for all $-1+2^{-k}\leq y\leq 1-2^{-k}$;
%$G_{j, k, \vec\poly}(x, 1-2^{-k})=\poly_j( \frac{x}{1-2^{-k}}, 1)=0$ and $G_{j,
%k, \vec\poly}(x, -1+2^{-k})=\poly_j( \frac{x}{1-2^{-k}}, -1)=0$ for all
%$-1+2^{-k}\leq x\leq 1=2^{=k}$,
$\Trim_k\vec\poly=0$ on $\partial\Omega_k$, for $\vec\poly$ vanishes on
the boundary of $\Omega$. Thus $\Trim_k\vec\poly$
vanishes in the exterior region of $\Omega_k$ including its boundary
$\partial\Omega_k$.
\item[e)]  $\Trim_k\vec\poly$ is
divergence-free in $\Omega_k$ following the calculation below: for
$(x, y)\in \Omega_k$, we have $(\frac{x}{1-2^{-k}},
\frac{y}{1-2^{-k}})\in \Omega$ and
\begin{align*}
\frac{\partial \Trim_k\poly_1}{\partial x}(x, y)+ \frac{\partial
\Trim_k\poly_2}{\partial y}(x, y)
& =  \frac{1}{1-2^{-k}}\frac{\partial \poly_1}{\partial x'}(x', y')+
\frac{1}{1-2^{-k}}\frac{\partial \poly_2}{\partial
y'}(x', y') \\
& =  \frac{1}{1-2^{-k}}\left[ \frac{\partial \poly_1}{\partial
x'}(x', y')+\frac{\partial \poly_2}{\partial y'}(x',y')\right] \quad=\; 0
\end{align*}
for $\vec\poly$ is divergence-free in $\Omega$, where
$x'=\frac{x}{1-2^{-k}}$ and $y'=\frac{y}{1-2^{-k}}$.
\end{enumerate}

\medskip\noindent
It follows from the discussion above that every $\Trim_k\vec\poly$
is a divergence-free polynomial of rational coefficients on
$\Omega_k$ that vanishes in $[-1, 1]^2\setminus \Omega_k$ and is
continuous on $[-1, 1]^2$. However, although the functions
$\Trim_k\vec\poly$ are continuous on $[-1, 1]^2$ and differentiable
in $\Omega_k$, they can be non-differentiable along the boundary
$\partial\Omega_k\subseteq\Omega$. To use these functions as names
for coding elements in $\SOLZ{\p}(\Omega)$, it is desirable to
smoothen them along the boundary $\partial \Omega_k$ so that they are
differentiable in the entire $\Omega$. A standard technique for
smoothing a function is to convolute it with a $\Cinfty$ function.
We use this technique to modify functions $\Trim_k\vec\poly$ so that
they become divergence-free and differentiable on the entire region
of $\Omega$. Let
\begin{equation} \label{e:gamma}
\gamma(\vec x) \;:=\; \left\{ \begin{array}{ll}
                                  \gamma_0 \cdot\exp\Big(-\tfrac{1}{1-\|\vec x\|^2}\Big), & \text{if } 1>\|\vec x\|:=\max\{|x_1|, |x_2|\} \\
                                  0, & \text{otherwise} \end{array}
                                  \right.
\end{equation}
where $\gamma_0$ is a constant such that $\int _{\IR^2}\gamma(\vec
x)\,d\vec x=1$ holds. The constant $\gamma_0$ is computable, since
integration on continuous functions is computable
\mycite{\S6.4}{Weih00}. Let $\gamma _{k}(\vec x)=2^{2k}\gamma
(2^{k}\vec x)$. Then, for all $k\in \IN$, $\gamma _{k}$ is a
$\Cinfty$ function having support in the closed square $[-2^{-k},
2^{-k}]^2$ and $\int _{\IR^2}\gamma _{k}(\vec x)\,d\vec x=1$. Recall
that for differentiable functions $f,g:\IR^2\to\IR$ with compact
support, the convolution $f\ast g$ is defined as follows:
\begin{equation} \label{e:convolution}
\big(f\ast g\big)(\vec x)\;=\;
  \int_{\IR^2} f(\vec x-\vec y)\cdot g(\vec y)\,d\vec y
\end{equation}
It is easy to see that for $n\geq k+1$ the support of
$\gamma_n\ast\Trim_k\vec\poly:=(\gamma_n\ast\Trim_k\poly_1,\gamma_n\ast\Trim_k\poly_2)$
is contained in the closed square $[-1+2^{-(k+1)}, 1-2^{-(k+1)}]^2$.
It is also known classically that $\gamma_n\ast\Trim_k\vec\poly$ is
a $\Cinfty$ function. Since $\gamma_n$ is a computable function and
integration on compact domains is computable, the map
$(n,k,\vec\poly)\mapsto \gamma_n\ast\Trim_k\vec\poly$ is computable.
Moreover the following metric is computable:
\begin{equation}
\label{e:Metric}
\big((n,k,\vec\poly),(n',k',\vec\poly')\big)
\;\mapsto\; \left( \int
\big|(\gamma_n\ast\Trim_k\vec\poly)(\vec x)-(\gamma_{n'}\ast\Trim_{k'}\vec\poly')(\vec x)\big|^2 \,d\vec x \right)^{1/2}
\end{equation}

%Moreover, for each $n\in\IN$, the Laplacian $\Laplace \gamma_n$ is
%also computable and $\Cinfty$. Finally note that
%\begin{equation} \label{e:mollified}
%\Laplace (\gamma_n\ast \Trim_k\vec\poly)\;=\;(\Laplace
%\gamma_n)\ast\Trim_k\vec\poly
%, \quad
%\divergence (\gamma_n\ast\Trim_k\vec\poly)
%\;=\; \gamma_n\ast\divergence\Trim_k\vec\poly \; = \vec 0
%\text{, and}\quad
%\divergence \Laplace (\gamma_n\ast\Trim_k\vec\poly)
%\;=\; (\Laplace\gamma_n)\ast\divergence\Trim_k\vec\poly \; = 0.
%\end{equation}

\begin{lemma} \label{l:convolution}
Every function $\gamma_{n}\ast\Trim_k\vec\poly$ is divergence-free
in $\Omega$, where $n,k\in\IN$, $n\geq k$, and $\vec\poly\in
\SOLPZ[\IR^2]$.
\end{lemma}
\begin{lemma} \label{l:dense}
The set $\calP=\big\{ \gamma_{n}\ast\Trim_k\vec\poly \::\: n,
k\in\IN, n\geq k+1, \vec{\poly}\in \SOLPZ[\IR^2]\big\}$ is dense in
$\SOLZ{\p}(\Omega)$.
\end{lemma}
See Appendices \ref{a:convolution} and \ref{a:dense} for the proofs.

\bigskip
\noindent From Lemmas~\ref{l:convolution} and \ref{l:dense} it
follows that $\calP$ is a countable set that is dense in
$\SOLZ{\p}(\Omega)$ (in $\ELL{\p}$-norm) and every function in
$\calP$ is $\Cinfty$, divergence-free on $\Omega$, and having a
compact support contained in $\Omega$; in other words, $\calP\subset
\{ \vec u\in \Cinfty_{0}(\Omega)^2 \, : \, \divergence \vec u=0\}$.
Thus, $\SOLZ{\p}(\Omega)=\mbox{the closure of $\calP$ in
$\ELL{\p}$-norm}$. This fact indicates that the set $\calP$ is
qualified to serve as codes for representing $\SOLZ{\p}(\Omega)$.
%But, in order to get a ``computable
%name" for the Stokes operator (the definition will be given
%shortly), we shall expand $\calP$ as follows:
%Let
%\[ \frakP \;=\; \big \{ \gamma_{n+1}\ast\Trim_n\vec\poly, \,
%\, (-\Laplace \gamma_{n+1})\ast\Trim_n\vec\poly\, : \,
%\vec\poly\in\SOLPZ[\vec X], n\in\IN\big\} \enspace .
%\]
%Then $\frakP$ is still
%countable and dense in $\SOLZ{\p}(\Omega)$. It is this set $\frakP$
%which we use as a set of ``codes." The desired representation is
%presented in the following paragraph.\newline

Since the function $\phi:
\bigcup_{N=0}^{\infty}\IQ^{(N+1)^2}\times
\IQ^{(N+1)^2}\to \{ 0, 1\}$, where
\[ \phi \big((r_{i, j})_{0\leq i, j\leq N}, (s_{i, j})_{0\leq i, j\leq N}\big)=\left\{
\begin{array}{ll}
                               1, & \mbox{if (\ref{div}),
(\ref{boundary-1}), and (\ref{boundary-2}) are satisfied}  \\
& \text{ (with
$r_{i, j}=a^1_{i, j}$ and $s_{i, j}=a^2_{i, j}$)} \\
                               0, & \mbox{otherwise} \end{array}
\right. \] is computable, there is a total computable function on
$\IN$ that enumerates $\SOLPZ[\IR^2]$. Then it follows from the
definition of $\calP$ that there is a total computable function
$\alpha: \IN\to \calP$ that enumerates $\calP$; thus,
in view of the computable equation~(\ref{e:Metric}),
$\big(\SOLZ{\p}(\Omega), (\vec u,\vec v)\mapsto\|\vec u-\vec
v\|_{\p}, \calP, \alpha\big)$ is a computable metric space. Let
$\deltaSOLZ{\p}: \IN^{\omega}\to \SOLZ{\p}$ be the standard Cauchy
representation of $\SOLZ{\p}$; that is, every function $\vec u\in
\SOLZ{\p}(\Omega)$ is encoded by a sequence $\{ \vec
\poly_k:k\in\IN\}\subseteq \calP$, such that $\|\vec u-\vec
\poly_k\|_{\p}\leq 2^{-k}$. The sequence $\{ \vec \poly_k\}_{k\in
\IN}$ is called a $\deltaSOLZ{\p}$-name of $\vec u$, which is an
effective approximation of $\vec u$ (in $L_2$-norm).

\section{Computability of Helmholtz Projection} \label{s:Helmholtz}
In this section, we show that the Helmholtz projection $\Helmholtz$
is computable.

\begin{proposition} \label{projection-P}
The projection $\Helmholtz : \big(\ELL{2}(\Omega)\big)^2
\to \SOLZ{2}(\Omega)$ is
$\big((\deltaELL{2})^2,\deltaSOLZ{2}\big)$-computable.
\end{proposition}

\begin{proof}
For simplicity let us set $\Omega = (0, 1)^2$. The proof carries
over to $\Omega=(-1, 1)^2$ by a scaling on sine and cosine
functions. We begin with two classical facts which are used in the
proof:
\begin{itemize}
\item[(i)] It follows from
\cite[pp.40]{GiRa86}/\cite{Tema77} that for any $\vec u=\big(u_{1},
u_{2}\big)\in \big(\ELL{2}(\Omega)\big)^2$,
\begin{equation} \label{P-u}
\Helmholtz\vec u=%\left( \begin{array}{c} -\partial_{y}\varphi \\
%\partial_{x}\varphi \end{array} \right) =
(-\partial_{y}\varphi,
\partial_{x}\varphi)
\end{equation}
where the scalar function $\varphi$ is the solution of the following
boundary value problem:
\begin{equation} \label{varphi}
\Laplace \varphi = \partial_{x}u_{2}-\partial_{y}u_{1} \text{ in }
\Omega, \qquad \varphi = 0 \text{ on } \partial \Omega
\end{equation}
We note that $\Helmholtz$ is a linear operator.
\item[(ii)] Each of $\{ \sin (n\pi x)\sin (m\pi y)\}_{n, m\geq 1}$,
$$\{ \sin (n\pi x)\cos (m\pi y)\}_{n\geq 1, m\geq 0}, \ \mbox{or}\
\{ \cos (n\pi x) \sin (m\pi y)\}_{n\geq 0, m\geq 1}$$ is an
orthogonal basis for $L^2(\Omega)$. Thus each $\vec u=\big(u_{1},
u_{2}\big)$ in $\big(\ELL{2}(\Omega)\big)^2$, $u_{i}$, $i=1$ or $2$,
can be written in the following form:
\begin{align*}
u_{i}(x, y)\;= & \;\sum\limits _{n, m\geq 0}u_{i, n, m}\cos(n\pi
x)\,\sin(m\pi y) \\
 \;= & \; \sum\limits_{n, m\geq 0}\tilde{u}_{i, n,
m}\sin(n\pi x)\,\cos(m\pi y)
\end{align*}
where
\[ u_{i, n,
m}=\int^1_0\int^1_0u_{i}(x, y)\cos(n\pi x)\sin(m\pi y)dxdy, \quad
\mbox{and}
\]
\[ \tilde{u}_{i, n, m}=\int^1_0\int^1_0u_{i}(x, y)\sin (n\pi x)\cos
(m\pi y)dxdy \enspace  \]
\end{itemize}
with the property that $\| u_{i}\|_{2} = \left( \sum_{n, m\geq
0}|u_{i, n, m}|^2\right)^{1/2} = \left( \sum_{n, m\geq
0}|\tilde{u}_{i, n, m}|^2\right)^{1/2}$. We note that the sequences
$\{ u_{i, n, m}\}$,  $\{ \tilde{u}_{i, n, m}\}$, and $\|
u_{i}\|_{2}$ are computable from $\vec{u}$; cf. \cite{Weih00}.

To prove that the projection is $\big((\deltaELL{2})^2,
\deltaSOLZ{2}\big)$-computable, it suffices to show that there is an
algorithm computing, given any accuracy $k\in\mathbb{N}$ and for any
$\vec{u}\in (L^2(\Omega))^2$, a vector function $(p_k, q_k) \in
\mathcal{P}$ such that $\| \Helmholtz\vec u - (p_k, q_k)\|_{2}\leq
2^{-k}$. Let us fix  $k$ and $\vec u=\big(u_{1}, u_{2})$. Then a
straightforward computation shows that the solution $\varphi$ of
(\ref{varphi}) can be explicitly written as
\[
\varphi(x, y)=\sum\nolimits_{n, m=1}^{\infty}\frac{-nu_{2, n,
m}+m\tilde{u}_{1, n, m}}{(n^2+m^2)\pi}\sin(n\pi x)\,\sin(m\pi y)
\]
It then follows that
\begin{equation} \label{partial-y-varphi}
-\partial_{y}\varphi = \sum\nolimits_{n, m\geq 1}\frac{mnu_{2, n,
m}-m^2\tilde{u}_{1, n, m}}{n^2+m^2}\sin(n\pi x)\,\cos(m\pi y)
\end{equation}
Similarly, we can obtain a formula for $\partial_{x}\varphi$. %Moreover,
%it follows from (\ref{P-u}) that $(-\partial_{y}\varphi,
%\partial_{x}\varphi)=\mathbb{P}\vec{u} \in L^{\sigma}_{2,
%0}(\Omega)$.
Since we have an explicit expression for
$(-\partial_{y}\varphi, \partial_{x}\varphi)$, a search algorithm is
usually a preferred choice for finding a $k$-approximation $(p_k,
q_k)$ of $\mathbb{P}\vec{u}$ by successively computing the norms $$\|
(-\partial_{y}\varphi,
\partial_{x}\varphi) - (p, q)\|_2, \qquad (p, q)\in \mathcal{P}.$$
However, since $-\partial_{y}\varphi$ and $\partial_{x}\varphi$ are
infinite series which involve limit processes, a truncating
algorithm is needed so that one can compute approximations of the
two limits before a search program can be executed. The truncating
algorithm will find some $N(k, \vec{u})\in \mathbb{N}$ such that the
$N(k, \vec{u})$-partial sum of $(-\partial_{y}\varphi,
\partial_{x}\varphi)$ is a $2^{-(k+1)}$-approximation of the series;
in other words, the algorithm chops off the infinite tails of the
series within pre-assigned accuracy. The following estimate provides
a basis for the desired truncating algorithm:
\begin{eqnarray*}
& & \| -\partial_{y}\varphi - \sum_{n, m<N}\frac{mnu_{2, n,
m}-m^2\tilde{u}_{1, n, m}}{n^2+m^2}\sin (n\pi
x)\cos (m\pi y)\|^2_{2} \\
& = & \| \sum_{n, m \geq N}\frac{mnu_{2, n, m}-m^2\tilde{u}_{1, n,
m}}{n^2+m^2}\sin (n\pi
x)\cos (m\pi y)\|^2_{2} \\
& = & \sum_{n, m\geq N}\left| \frac{mnu_{2, n, m}-m^2\tilde{u}_{1,
n, m}}{n^2+m^2}\right|^2 \leq 2\sum_{n, m\geq N}(|u_{2, n, m}|^2 +
|\tilde{u}_{1, n, m}|^2) \\
%& = & \left| (u^{(1)}, u^{(2)}) - \left( \sum_{n, m<N}\tilde{u}_{1,
%n, m}\sin (n\pi x)\cos (m\pi y), \sum_{n, m<N}u_{2, n, m}\cos (n\pi
%x)\sin (m\pi y)\right) \right| ^2
\end{eqnarray*}
A similar estimate applies to $\partial_{x}\varphi$. Since
$$\|
u_{i}\|^{2}_{2}=\sum_{n, m\geq 1}|u_{i, n, m}|^2 = \sum_{n, m\geq
1}|\tilde{u}_{i, n, m}|^2\, ,\quad  i=1, 2,
$$
is computable, it follows that
there is an algorithm computing $N(k, \vec{u})$ from $k$ and
$\vec{u}$ such that the $N(k, \vec{u})$-partial sum of
$(-\partial_{y}\varphi, \partial_{x}\varphi)$ is a
$2^{-(k+1)}$-approximation of the series. Now we can search for
$(p_k, q_k)\in \mathbb{P}$ that approximates  the $N(k,
\vec{u})$-partial sum in $L^2$-norm within the accuracy $2^{-(k+1)}$
as follows: enumerate $\mathcal{P}=\{ \tilde{p}_j\}$, compute the
$L^2$-norm of the difference between the $N(k, \vec{u})$-partial sum
and $\tilde{p}_j$, halt the computation at $\tilde{p}_j$ when the
$L^2$-norm is less that $2^{-(k+1)}$, and then set $(p_k,
q_k)=\tilde{p}_j$. We note that each computation halts in finitely
many steps.  The search will succeed since
$\mathbb{P}\vec{u}=(-\partial_{y}\varphi,
\partial _{x}\varphi)\in L^{\sigma}_{2, 0}$ and $\mathcal{P}$ is dense in $L^{\sigma}_{2,
0}$. It is then clear that $\| \mathbb{P}\vec{u} - (p_k, q_k)\|_2
\leq 2^{-k}$.
\end{proof}

%%%%%%%%%%%%%%%%%%%%%%%%%%%%%%%%%%%%%%%%%%%%%%%%%%%%%%%%%%
\section{Computability of the linear problem} \label{s:Stokes}

In this section, we show that the solution operator for the linear
homogeneous equation (\ref{e:NS-LH}) is uniformly computable from
the initial data. We begin by recalling the Stokes operator and some
of its classical properties. Let $\IA =-\Helmholtz \Laplace$ be the
Stokes operator as defined for instance in \cite[\S2]{Giga83a} or
\mycite{\S III.2.1}{Sohr}, where $\Helmholtz :
\big(\ELL{\p}(\Omega)\big)^2 \to \SOLZ{\p}$ is the Helmholtz
projection. It is known from the classical study that $\IA $ is an
unbounded but closed positively self-adjoint linear operator whose domain is dense
in $\SOLZ{\p}$, and thus $-\IA $ is the infinitesimal generator of
an analytic semigroup; cf. \mycite{Theorem~III.2.1.1}{Sohr} or
\cite[\S IV.5.2]{BoyerFabrie}. In this case, the linear homogeneous
equation (\ref{e:NS-LH}) has the solution $\vec u(t)=e^{-\IA t}\vec
a$, where $\vec u(0)=\vec a$, $e^{-\IA t}$ is the analytic semigroup
generated by the infinitesimal generator $-\IA $, and $\vec{u}(t)\in
L^{\sigma}_{2, 0}(\Omega)$ for $t\geq 0$. Furthermore, the following
lemma shows that the solution $\vec{u}(t)$ decays in $L^2$-norm as
time $t$ increases.

\begin{lemma} \label{l:linear-estimate}
For every $\vec{a}\in L^{\sigma}_{2, 0}(\Omega)$ and $t\geq 0$,
\begin{equation} \label{C-0-I}
\| \vec{u}(t)\|_2 = \| e^{-t\IA}\vec a\|_{2}\leq \|\vec a\|_{2}=\|
\vec{u}(0)\|_2
\end{equation}
(Recall that $\| \cdot \|_{2}=\| \cdot \|_{\SOLZ{2}(\Omega)}$; see
Notation \ref{n:1}.)
\end{lemma}

\begin{proof} Classically it is known that for any $\vec{a}\in L^{\sigma}_{2,
0}(\Omega)$, $\vec{u}(t)=e^{-t\IA}\vec a$ is in the domain of $\IA$
for $t>0$. Thus if $\vec a=\vec{u}(0)$ itself is in the domain of
$\IA$, then so is $\vec{u}(t)$ for $t\geq 0$. Since $\IA$ is
positively self-adjoint, it follows that $\IA ^{\ast}=\IA $ and
$\langle\IA\vec u(t), \vec u(t)\rangle:=\int_{\Omega} \IA\vec
u(t)(\vec x)\cdot\vec u(t)(\vec x)\,d\vec x > 0$ for every $\vec a$
in the domain of $\IA$ with $\vec a\not\equiv 0$ and $t\geq 0$. Now
if we rewrite the equation (\ref{e:NS-LH}) in the form of
\[ \langle \vec u_t, \vec u \rangle + \langle \IA\vec u, \vec u \rangle=0 \]
or equivalently $\frac{1}{2}\frac{d}{dt}\langle \vec u, \vec
u\rangle + \langle \IA\vec u, \vec u \rangle=0$, then
$\frac{d}{dt}\langle \vec u, \vec u \rangle \leq 0$ and consequently
$\langle \vec u, \vec u \rangle (t)\leq \langle \vec u, \vec u
\rangle (0)$; thus (\ref{C-0-I}) holds true for $\vec{a}$ in the
domain of $\IA$. Since the domain of $\IA$ is dense in
$L^{\sigma}_{2, 0}(\Omega)$, it follows that (\ref{C-0-I}) holds
true for all $\vec{a}\in L^{\sigma}_{2, 0}(\Omega)$.
\qed\end{proof}

\begin{proposition} \label{p:linear-NS}
For the linear homogenous equation~(\ref{e:NS-LH}), the solution
operator $\Semigroup:\SOLZ{2}(\Omega)\to\C\big([0;\infty),
\SOLZ{2}(\Omega)\big)$, $\vec a\mapsto (t\mapsto e^{-\IA t}\vec a)$,
is $(\deltaSOLZ{2}, [\myrho \to \deltaSOLZ{2}])$-computable.
\end{proposition}

By the \emph{First Main Theorem} of Pour-El and Richards \cite[\S II.3]{PERi89},
the unbounded operator $\IA$ does not preserve computability.
In particular, the naive exponential series $\sum_n (-\IA t)^n \vec a/n!$
does not establish Proposition~\ref{p:linear-NS}.
\bigskip

\noindent \textbf{Convention.} For readability we will not notationally
distinguish the spaces of vectors $\vec u,\vec a$ and scalar functions $u,a$
in the proof below and the proof of Lemma ~\ref{H-alpha}. \\

\begin{proof}  We show how to
compute a $\deltaSOLZ{2}$-name of $e^{-t\IA}a$ on inputs $t\geq 0$
and  $a\in \SOLZ{2}(\Omega)$. Recall that a $\deltaSOLZ{2}$-name of
$e^{-t\IA}a$ is a sequence $\{ q_K\}$, $q_K\in \mathcal{P}$,
satisfying $\|e^{-t\IA}a - q_K\|_{2}\leq 2^{-K}$ for all
$K\in\mathbb{N}$. Again, for readability, we assume that $\Omega=(0,
1)^2$.

We first consider the case where $a\in \mathcal{P}$ and $t>0$. The
reason for us to start with functions in $\mathcal{P}$ is that these
functions have stronger convergence property in the sense that, for
any $a\in \mathcal{P}$, if  $a=(a^1, a^2)$ is expressed in terms of
the orthogonal basis $\{ \sin(n\pi x)\sin(m\pi y)\}_{n, m\geq 1}$
for $\ELL{2}(\Omega)$: for $i=1, 2$,
\begin{equation} \label{Fourier-series}
a^i=\sum_{n, m\geq 1}a^i_{n, m}\sin(n\pi x)\sin(m\pi y)
\end{equation}
where $a^i_{n, m}=\int_0^1\int_0^1 a^i\sin(n\pi x)\sin(m\pi
y)\,dx\,dy$, then the following series is convergent
\begin{equation} \label{H^1-norm}
\sum_{n, m\geq 1}(1+n^2+m^2) |a^i_{n, m}|^2 < \infty
\end{equation}
The inequality (\ref{H^1-norm}) holds true because functions in
$\mathcal{P}$ are $\Cinfty$. In fact, the series is not only
convergent but its sum is also computable (from $a$) (see, for
example, \cite{Zhon99}).

Now let $K\in \mathbb{N}$ be any given precision. Since $-\IA$
generates an analytic semigroup, it follows from \mycite{Section
2.5}{Pazy83} that for $t>0$,
\begin{equation} \label{e:Pazy}
e^{-t\IA}a=\frac{1}{2\pi i}\int_{\Gamma}e^{\lambda t}
(\lambda\II+\IA)^{-1}a\,d\lambda
\end{equation}
where $\Gamma$ is the path composed from two rays $re^{i\beta}$ and
$re^{-i\beta}$ with $0<r<\infty$ and $\beta=\frac{3\pi}{5}$. Thus we
have an explicit expression for $e^{-t\IA}a$, which involves a limit
process -- an infinite integral -- indicating that a search
algorithm is applicable  for finding a desirable $K$-approximation
provided that a finite approximation of $e^{-t\IA}a$ can be computed
by some truncating algorithm.

In the following, we construct such a truncating algorithm. We begin
by writing the infinite integral in (\ref{e:Pazy}) as a sum of three
integrals: two are finite and one infinite; the infinite one can be
made arbitrarily small. Now for the details. Let $l$ be a positive
integer to be determined; let $\Gamma_1$ be the path $re^{i\beta}$
with $0<r\leq l$; $\Gamma_2$ the path $re^{-i\beta}$ with $0<r\leq
l$; and $\Gamma_3=\Gamma \setminus (\Gamma_1\cup\Gamma_2)$. Since
$a\in \mathcal{P}$, it follows that $-\IA a=\IP \Laplace a=\Laplace
a$, which further implies that
\begin{eqnarray} \label{resolvent-a}
& (\lambda\II+\IA)^{-1} a =  \nonumber \\
& \left(\sum_{n, m\geq 1}\frac{a^1_{ n, m}\sin(n\pi x)\sin(m\pi
y)}{\lambda + (n\pi)^2+(m\pi)^2}, \sum_{n, m\geq 1}\frac{a^2_{ n,
m}\sin(n\pi x)\sin(m\pi y)}{\lambda + (n\pi)^2+(m\pi)^2}\right)
\end{eqnarray}
Note that for any $\lambda\in \Gamma$, $|\lambda +
(n\pi)^2+(m\pi)^2|\neq 0$. From (\ref{e:Pazy}) and
(\ref{resolvent-a}) we can write $e^{-t\IA}a$ as a sum of three
terms:
\begin{eqnarray*}
e^{-t\IA}a  & = & \sum_{j=1}^{3}\frac{1}{2\pi
i}\int_{\Gamma_j}\tilde{a}e^{\lambda t}d\lambda
\\
& = & \sum_{j=1}^3\frac{1}{2\pi i}\sum_{n, m\geq 1}\left[
\int_{\Gamma_j}\frac{e^{\lambda t}}{\lambda +
(n\pi)^2+(m\pi)^2}d\lambda \right] a_{n, m}\sin (n\pi x)\sin(m\pi
y) \\
& =: & \beta_1+\beta_2+\beta_3 \enspace
\end{eqnarray*}
where $\tilde{a}=(\lambda\II+\IA)^{-1} a$. The functions $\beta_j$,
$j=1, 2, 3$, are in $L^{\sigma}_{2,0}(\Omega)$ as verified as
follows: It follows from $a=(\lambda\II+\IA)\tilde{a} =
(\lambda\II-\IP\bigtriangleup)\tilde{a}$ and
$\IP\bigtriangleup\tilde{a}=\IP (\bigtriangleup \tilde{a})\in
\SOLZ{2}(\Omega)$ that $\bigtriangledown
(\IP\bigtriangleup\tilde{a})=0$ and
\begin{equation} \label{tilde-a}
0=\bigtriangledown a=\lambda (\bigtriangledown \tilde{a}) -
\bigtriangledown (\IP\bigtriangleup \tilde{a})=\lambda
(\bigtriangledown \tilde{a})
\end{equation}
Since $\lambda \in \Gamma$, it follows that $\lambda \neq 0$; thus
$\bigtriangledown \tilde{a}=0$. This shows that $\tilde{a}\in
\SOLZ{2}(\Omega)$. Then it follows from  (\ref{tilde-a}) that
\[ \bigtriangledown\beta_j = \frac{1}{2\pi
i}\int_{\Gamma_j}(\bigtriangledown \tilde{a})e^{\lambda t}d\lambda =
0\] Hence $\beta_j\in \SOLZ{2}(\Omega)$ for $1\leq j\leq 3$.

Next we show that $\beta_1$ and $\beta_2$ can be effectively
approximated by finite sums while $\beta_3$ tend to zero effectively
as $l\to \infty$. We start with $\beta_3$. Since $t>0$ and
$\cos\beta =\cos\frac{3\pi}{5}<0$, it follows that
\[ \left| \int_{\Gamma_3}\frac{e^{\lambda
t}}{\lambda+(n\pi)^2+(m\pi)^2}d\lambda\right| \leq
2\int_l^{\infty}\frac{e^{tr\cos \beta}}{r}dr \to 0
\] effectively as $l\to \infty$.
%This fact together with
%(\ref{H^1-norm}) implies that
Thus there is some $l_K\in \IN$, computable from $a$ and $t$, such
that the following estimate is valid for $i=1, 2$ when we take $l$
to be $l_K$:
\begin{eqnarray*}
& & \| \beta^i_3\|_{2} \\
%& \leq & \|  \beta_3 \|^2_{\SOB{2}{1}}
&=& \left\|\frac{1}{2\pi i}\sum_{n, m\geq 1}\left[
\int_{\Gamma_3}\frac{e^{\lambda t}}{\lambda +
(n\pi)^2+(m\pi)^2}d\lambda \right] a^i_{n, m}\sin (n\pi
x)\sin(m\pi y)\right\|_{2} \\
& \leq & \frac{1}{\pi}\int^{\infty}_{l_K}\frac{e^{tr\cos
\beta}}{r}dr \left(\sum_{n, m\geq 1}|a^i_{n, m}|^2\right)^{1/2} =
\frac{1}{\pi}\int^{\infty}_{l_K}\frac{e^{tr\cos \beta}}{r}dr \cdot
\| a\|_{2} \leq 2^{-(K+7)} \enspace
\end{eqnarray*}
where $\beta_3 = (\beta^1_3, \beta^2_3)$. Now let us set $l=l_K$ and
estimate $\beta_1$. Since $\beta =\frac{3\pi}{5}<\frac{3\pi}{4}$, it
follows that $\cos \beta <0$ and $|\cos\beta|<\sin \beta$.
Consequently, for any $\lambda =re^{i\beta}$ on $\Gamma_1$, if
$r\geq \frac{1}{\sin\beta}$, then
$|re^{i\beta}+(n\pi)^2+(m\pi)^2|\geq r\sin \beta\geq 1$. On the
other hand, if $0<r<\frac{1}{\sin\beta}$, then $r\cos\beta
+(n\pi)^2+(m\pi)^2 \geq \pi^2(n^2+m^2)-r\sin\beta \geq 2\pi^2 - 1
>1$, which implies that $|re^{i\beta}+(n\pi)^2+(m\pi)^2|\geq
|r\cos\beta +(n\pi)^2+(m\pi)^2|>1$. Thus $|\lambda
+(n\pi)^2+(m\pi)^2|\geq 1$ for every $\lambda \in \Gamma_1$. And so
\begin{multline*}
\bigg | \int_{\Gamma_1}  \frac{e^{\lambda
t}}{\lambda+(n\pi)^2+(m\pi)^2}d\lambda\bigg |   =  \left|
\int_0^l\frac{e^{tre^{i\beta}}}{re^{i\beta}+(n\pi)^2+(m\pi)^2}d(re^{i\beta})\right|
\;\leq \\
 \leq \;
\int_0^l\frac{|e^{tre^{i\beta}}|}{|re^{i\beta}+(n\pi)^2+(m\pi)^2|}dr
\leq  \int_0^l e^{tr\cos \beta}dr \leq  \int_0^le^{tl}dr = le^{tl}
\end{multline*}
This estimate together with (\ref{H^1-norm}) implies that there
exists a positive integer $k=k(t, a, K)$, computable from $t>0$, $a$
and $K$, such that
\[ \frac{1}{1+2k^2}\left(\frac{le^{lt}}{2\pi}\right)^2\left(
\sum_{n, m\geq 1}(1+n^2+m^2)(|a^1_{n, m}|^2+|a^2_{n, m}|^2)\right) <
2^{-2(K+7)}
\]
Write $\beta_1(k)=(\beta^1_1(k), \beta^2_1(k))$ with
\[ \beta^i_1(k)=\sum_{1\leq n, m\leq k}\left( \frac{1}{2\pi
i}\int_{\Gamma_1}\frac{e^{\lambda t}}{\lambda + (n\pi)^2 +
(m\pi)^2}d\lambda\right)a^i_{n, m}\sin(n\pi x)\sin(m\pi y),
\]
$i=1, 2$. Then
\begin{eqnarray*}
& & \| \beta_1 - \beta_1(k)\|^2_{2} \\
& \leq & \sum_{n, m>k}\left| \frac{1}{2\pi
i}\int_{\Gamma_1}\frac{e^{\lambda t}}{\lambda +
(n\pi)^2+(m\pi)^2}d\lambda\right|^2 (|a^1_{n, m}|^2+|a^2_{n, m}|^2) \\
& \leq & \sum_{n,m >k}\frac{1}{1+n^2+m^2}\cdot
(1+n^2+m^2)\left(\frac{le^{lt}}{2\pi}\right)^2(|a^1_{n,
m}|^2+|a^2_{n, m}|^2) \\
& \leq & \frac{1}{1+k^2+k^2}\left(\frac{le^{lt}}{2\pi}\right)^2
\sum_{n,m \geq 1}(1+n^2+m^2)(|a^1_{n,
m}|^2+|a^2_{n, m}|^2) \\
& < & 2^{-2(K+7)}
\end{eqnarray*}
Similarly, if we write $\beta_2(k)=(\beta^1_2(k), \beta^2_2(k))$
with
\[ \beta^i_2(k)=\sum_{n, m\leq k}\left( \frac{1}{2\pi
i}\int_{\Gamma_2}\frac{e^{\lambda t}}{\lambda + (n\pi)^2 +
(m\pi)^2}d\lambda\right)a^i_{n, m}\sin(n\pi x)\sin(m\pi y) \]
then
$\| \beta_2 - \beta_2(k)\|_{2}\leq 2^{-(K+7)}$. The construction of
the truncating algorithm is now complete; the algorithm outputs
$\beta_1(k) + \beta_2(k)$ (uniformly) on the inputs $a\in
\mathcal{P}$, $t>0$, and precision $K$; the output has the property
that it is  a finite sum involving a finite integral and
$\|\beta_1(k)+\beta_2(k)-e^{-t\IA}a\|_{2}\leq 2^{-(K+4)}$.

Now we are able to search for a desirable approximation in
$\mathcal{P}$. Let us list $\mathcal{P}=\{ \phi_j \, : \,
j\in\mathbb{N}\}$ and compute $\| \phi_j -
(\beta_1(k)+\beta_2(k))\|_{2}$. Halt the computation at $j=j(K)$
when
\[ \| \phi_j
- (\beta_1(k)+\beta_2(k))\|_{2} < 2^{-(K+4)}
\] The computation will halt since $\beta_1, \beta_2\in
\SOLZ{2}(\Omega)$, $\| \beta_1 - \beta_1(k)\|_{2}\leq 2^{-(K+7)}$,
$\| \beta_2 - \beta_2(k)\|_{2}\leq 2^{-(K+7)}$, and $\mathcal{P}$ is
dense in $\SOLZ{2}(\Omega)$ (in $L^2$-norm). Set $q_K=\phi_{j(K)}$.
Then
\begin{eqnarray*}
& & \| q_K - e^{-t\IA}a\|_{2} \\
& = & \| q_K
-(\beta_1+\beta_2+\beta_3)\|_{2} \\
& \leq & \| q_K - (\beta_1(k)+\beta_2(k))\|_{2} +  \|
(\beta_1(k)+\beta_2(k)) - (\beta_1+\beta_2)\|_{2}
+ \|\beta_3\|_{2} \\
& < & 2^{-K}
\end{eqnarray*}

Next we consider the case where $a\in L^{\sigma}_{2, 0}(\Omega)$ and
$t>0$. In this case, the input $a$ is presented by (any) one of its
$\delta_{L^{\sigma}_{2, 0}}$-names, say $\{ a_k\}$, where $a_k\in
\mathcal{P}$. It is then clear from the estimate (\ref{C-0-I}) and
the discussion above that there is an algorithm that computes a
$K$-approximation $p_K\in \mathcal{P}$ on inputs $t>0$, $a$ and
precision $K$ such that $\| p_K - e^{-t\IA}a\|_{2}\leq 2^{-K}$.

Finally we consider the case where $t\geq 0$ and $a\in
L^{\sigma}_{2, 0}(\Omega)$. Since $e^{-t\IA}a=a$ for $t=0$ and we
already derived an algorithm for computing $e^{-t\IA}a$ for $t>0$,
it suffices to show that $e^{-t\IA}a \to a$ in $L^2$-norm
effectively as $t\to 0$. Let $\{ a_k\}$ be a
$\delta_{L^{\sigma}_{2,0}}$-name of $a$. It follows from Theorem
6.13 of Section 2.6 [Paz83] that $\| e^{-t\IA}a_k - a_k\|\leq
Ct^{1/2}\| \IA^{1/2}a_k\|$. Thus \[ \| a - e^{-t\IA}a\| \leq \| a -
a_k\| + \| a_k - e^{-t\IA}a_k\| + \| e^{-t\IA}a_k - e^{-t\IA}a\| \]
the right-hand side goes to 0 effectively as $t \to 0$.
\qed\end{proof}

We note that the computation of the approximations $q_K$ of
$e^{-t\IA}a$ does not require encoding $\IA$. Let $W:
\SOLZ{2}(\Omega)\times [0, \infty) \to \SOLZ{2}(\Omega)$, $(a,
t)\mapsto e^{-t\IA}a$. Then it follows from the previous Proposition
and Fact \ref{f:conversion} that $W$ is computable.

%%%%%%%%%%%%%%%%%%%%%%%%%%%%%%%%%%%%%%%%%%%%%%%%%%%%%%%%%%%%%%%%%%
\section{Extension to the nonlinear problem} \label{s:Nonlinear}
We now proceed to the nonlinear problem (\ref{e:NS-E}) by solving
its integral version (\ref{e:integral-form}) via the iteration
scheme (\ref{e:iteration}) but first restrict to the homogeneous
case $\vec g\equiv\vec0$:

\begin{equation} \label{e:iteration0}
\vec u_0(t)\;=\;e^{-t\IA}\vec a , \qquad \vec u_{m+1}(t)\;=\;\vec
u_0(t)\;-\;\int^t_0e^{-(t-s)\IA }\IB\vec u_m(s)\,ds
\end{equation}
Classically, it is known that for every initial condition $\vec
a\in\SOLZ{2}(\Omega)$ the sequence $\vec u_m=\vec u_m(t)$ converges
near $t=0$ to a unique limit $\vec u$ solving
(\ref{e:integral-form}) and thus (\ref{e:NS-E}). Since there is no
explicit formula for the solution $\vec{u}$, the truncation/search
type of algorithms such as those used in the proofs of Propositions
\ref{projection-P} and \ref{p:linear-NS} is no longer applicable for
the nonlinear case. Instead, we use a method based on the
fixed-point argument to establish the computability of $\vec u$. We
shall show that the limit of the above sequence $\vec u_m=\vec
u_m(t)$ has an effective approximation. The proof consists of two
parts: first we study the rate of convergence and show that the
sequence converges at a computable rate as $m\to \infty$ for
$t\in[0;T]$ with some $T=T_{\vec a}>0$, where $T_{\vec a}$ is
computable from $\vec a$; then we show that the sequence -- as one
entity -- can be effectively approximated starting with the given
$\vec a$. The precise statements of the two tasks are given in the
following two propositions.

\begin{proposition} \label{p:convergence} There is a computable map
$\mathbb{T}:\SOLZ{2}(\Omega)\to (0, \infty)$, $\vec{a}\mapsto
T_{\vec a}$, such that the sequence $\{ {\vec u}_m\}$ converges
effectively in $m$ and uniformly for $t\in[0;T_{\vec a}]$.
\end{proposition}

Recall that a sequence $\{x_m\}$ in a metric space $(X,d)$ is
effectively convergent if $d(x_m,x_{m+1})\leq2^{-m}$. In view of
type conversion (Subsection~\ref{ss:Overview}), the following
proposition asserts (ii):

\begin{proposition} \label{p:iteration0} The map
$\IS: \IN\times \SOLZ{2}(\Omega)\times [0, \infty)\to
\SOLZ{2}(\Omega)$, $(m, \vec a, t)\to \vec{u}_m(t)$ according to
Equation~(\ref{e:iteration0}), is
$\big(\nu\times\deltaSOLZ{2}\times\myrho,
\deltaSOLZ{2}\big)$-computable.
\end{proposition}

The main difficulties in proving the two propositions are rooted in
the nonlinearity of $\IB$: the nonlinear operator $\IB$ requires
greater care in estimating the rate of convergence and demands
richer codings for computation. Since information on $\IB\vec{u}_m$
is required in order to compute $\vec u_{m+1}$, but  $\IB\vec
u_m=\Helmholtz\, (\vec u_m\cdot\nabla)\vec u_m$ involves both
differentiation and multiplication, it follows that a
$\deltaSOLZ{2}$-name of $\vec{u}_m$ may not contain enough
information for computing $\IB \vec{u}_m$. Moreover, since estimates
of type $\| \IA^{\alpha}\vec{u}_m(t)\|_2$, $0\leq \alpha \leq 1$,
play a key role in proving Propositions \ref{p:convergence} and
\ref{p:iteration0}, we need to computationally derive a richer code
for $\vec{u}_m$ from a given $\delta_{L^{\sigma}_{2, 0}}$-name of
$\vec{u}_m$ in order to capture the fact that $\vec{u}_m$ is in the
domain of $\IA^{\alpha}$ for $t>0$.

%%%%%%%%%%%%%%%%%%%%%%%%%%%%%%%%%%%%%%%%%%%%%%%%%%%
\subsection{Representing and Operating on Space $H^s_{2, 0}(\Omega)$}
\label{ss:MulDiff}

We begin by recalling several definitions and facts. Let $\theta_{n,
m}(x, y):=e^{i ( nx + my)\pi}$, $n, m\geq 0$. Then, the sequence
$\{\theta_{n, m}(x, y)\}_{n, m\geq 0}$ is a  computable orthogonal
basis of $\ELL{2}(\Omega)$. For any $s\geq 0$, $\SOB{2}{s}(\Omega)$
is the set of all (generalized) functions $w(x, y)$ on $\Omega$
satisfying $\sum_{n,m\geq 0}(1+n^2+m^2)^s|w_{n,m}|^2<\infty$, where
$w_{n,m}=\int_{-1}^1\int_{-1}^1w(x, y)\theta_{n,m}(x, y)\,dx\,dy$.
$\SOB{2}{s}(\Omega)$ is a Banach space with a norm
$\|w\|_{\SOB{2}{s}}=(\sum_{n,m\geq
0}(1+n^2+m^2)^s|w_{n,m}|^2)^{1/2}$. \\

Let $D(\IA^{\alpha})$ be the domain of $\IA^{\alpha}$. Since
\begin{align*}
D(\IA)=& L^{\sigma}_{2, 0}(\Omega)\bigcap \{ \vec{u}\in
(H^2_2(\Omega))^2 : \mbox{$\vec{u}=0$ on $\partial \Omega$}\}, \\
D(\IA^{1/2})= & L^{\sigma}_{2, 0}(\Omega)\bigcap \{ \vec{u}\in
(H^1_2(\Omega))^2 : \mbox{$\vec{u}=0$ on $\partial \Omega$}\},
\end{align*}
and $D(\IA^{\alpha})$, $0\leq \alpha \leq 1$, are the complex
interpolation spaces of $L^{\sigma}_{2, 0}(\Omega)$ and $D(\IA)$, we
need to represent the subspace of $H^s_2(\Omega)$ in which the
functions vanish on $\partial \Omega$. However, it is usually
difficult to design a coding system for such subspaces. Fortunately,
for $0\leq s<3/2$, it is known classically that
\begin{equation} \label{e:s-restriction}
H^s_{2, 0}(\Omega) = \{ w\in H^s_{2}(\Omega) \, : \, \mbox{$w=0$ on
$\partial \Omega$}\}
\end{equation}
where $H^s_{2, 0}(\Omega)$ is
the closure in $H^s_2$-norm of the set of all $C^{\infty}$-smooth
functions defined on compact subsets of $\Omega$. For $H^s_{2,
0}(\Omega)$, there is a canonical coding system
\[ \mathcal{H} = \{ \gamma_n\ast q \, : \, n\in \mathbb{N}, q\in
\mathbb{Q}[\mathbb{R}^2] \} \] (see (\ref{e:gamma}) and
(\ref{e:convolution}) for the definitions of $\gamma_n$ and
$\gamma_n \ast q$). Then every $w$ in $H^s_{2, 0}(\Omega)$ can be
encoded by a sequence $\{ p_k\}\subset \mathcal{H}$ such that $\|
p_k - w\|_{H^s_2}\leq 2^{-k}$; the sequence $\{ p_k \}$, which are
mollified polynomials with rational coefficients,  is called a
$\delta_{H^s_{2, 0}}$-name of $w$. If $w=(w_1, w_2)\in H^s_{2,
0}(\Omega)\times H^s_{2, 0}(\Omega)$, a $\delta_{H^s_{2, 0}}$-name
of $w$ is a sequences $\{ (p_k, q_k)\}$, $p_k, q_k\in \mathcal{H}$,
such that $(\| w_1 - p_k\|^2_{H^s_{2, 0}} + \| w_2 -
q_k\|^2_{H^s_{2, 0}})^{1/2}\leq 2^{-k}$. \\

\begin{notation} \label{n:2} We use $\| w\|_{H^s_2}$ to denote the
$H^s_2$-norm of $w$ if $w$ is in $H^s_2(\Omega)$ or $H^s_2\times
H^s_2$-norm of $w$ if $w$ is in $H^s_2(\Omega)\times H^s_2(\Omega)$.
Also for readability we use $[\myrho\to \deltaSOB{s}]$ to denote the
canonical representation of either $\C\big([0;T];
\SOB{2,0}{s}(\Omega)\big)$ or $\C\big([0;T];
\SOB{2,0}{s}(\Omega)\times \SOB{2,0}{s}(\Omega)\big)$. \\
\end{notation}

Recall that $\C\big([0;T]; \SOB{2,0}{s}(\Omega)\big)$ is the set of
all continuous functions from the interval $[0;T]$ to
$\SOB{2,0}{s}(\Omega)$. A function $u\in \C\big([0;T];
\SOB{2,0}{s}(\Omega)\big)$ is computable if there is a machine that
computes a $\deltaSOB{s}$-name of $u(t)$ when given a $\myrho$-name
of $t$ as input; and a map $F: X\to \C\big([0;T];
\SOB{2,0}{s}(\Omega)\big)$ from a represented space $(X, \delta_X)$
to $\C\big([0;T]; \SOB{2,0}{s}(\Omega)\big)$ is computable if there
is a machine that computes a $\deltaSOB{s}$-name of $F(x)(t)$ when
given a $\delta_X$-name of $x$ and a $\myrho$-name of $t$. Let $X$
be either $\ELL{2}(\Omega)$, $L^{\sigma}_{2, 0}(\Omega)$,
$\SOB{2,0}{s}(\Omega)$, or $\C\big([0;T];
\SOB{2,0}{s}(\Omega)\big)$. We remark again that a $\delta_{X}$-name
of $f\in X$ is simply an effective approximation of $f$ because each
space $X$ is equipped with a norm. \\

\begin{lemma} \label{l:differentiation}
For $s\geq 1$, differentiation $\partial_x,\partial_y:
\SOB{2,0}{s}(\Omega)\to \ELL{2}(\Omega)$ is $(\deltaSOB{s},
\deltaELL{2})$-computable.
\end{lemma}

\begin{proof} Let $\{ p_k\}$ be a $\deltaSOB{s}$-name of
$w\in H^s_{2, 0}(\Omega)$. Since $\partial_x (\gamma\ast q) = \gamma
\ast \partial_xq$, the map $p_k \mapsto \partial_x p_k$ is
computable; hence a polynomial $\tilde{p}$ in
$\mathbb{Q}[\mathbb{R}^2]$ can be computed from $p_k$ such that
$\max_{-1\leq x, y\leq 1}|\partial _x p_k - \tilde{p}_k|<2^{-k}$.
Next let us express $w$ and $p_k$ in the orthogonal basis
$\theta_{n, m}$: $w(x, y)=\sum_{n,m\geq 0}w_{n, m}e^{in\pi
x}e^{im\pi y}$ and $p_k(x, y)= \sum_{n, m\geq 0}p_{k, n, m}e^{in\pi
x}e^{im\pi y}$, where
\begin{align*}
& w_{n,m}=\int_0^1\int_0^1w(x, y)e^{in\pi
x}e^{im\pi y}\,dx\,dy\, , \\
& p_{k, n, m}=\int_0^1\int_0^1p_k(x,
y)e^{in\pi x}e^{im\pi y}\,dx\,dy\, .
\end{align*}
 Since $s\geq 1$ and $\{ p_k\}$ is
a $\deltaSOB{s}$-name of $w$, it follows that
\begin{align*}
& \|\partial_x p_k - \partial_x w\|^2_2 \\
& =  \left\|
\sum\nolimits_{n,m}in\pi (p_{k,n,m}-w_{n,m})e^{in\pi x}e^{im\pi
y}\right\|^2_2
\; =\;  \pi^2\sum\nolimits_{n,m}n^2|p_{k, n, m}-w_{n, m}|^2 \\
& =  \pi^2\sum_{n, m}\frac{n^2}{(1+n^2+m^2)^s} (1+n^2+m^2)^s|p_{k,
n,
m}-w_{n, m}|^2 \\
& \leq  \pi^2\sum\nolimits_{n, m}(1+n^2+m^2)^s|p_{k, n, m}-w_{n,
m}|^2 \;=\; \pi^2\| p_k-w\|^2_{\SOB{2}{s}} \;\leq\; \pi^2\cdot
2^{-2k}
\end{align*}
which further implies that
\[ \| \tilde{p}_k - \partial_x w\|_2 \leq \| \tilde{p}_k - \partial_x
p_k\|_2 + \| \partial_x p_k - \partial_x w\|_2 \leq 2^{-k} + \pi
2^{-k} \] Thus, by definition, $\{ \tilde{p}_k \}$ is a
$\deltaELL{2}$-name of $\partial_x w$.
\end{proof}
\medskip

It is known classically that every polygonal domain in
$\IR^2$ is Lipschitz (see, for example, \cite{McL00}) and
$\SOB{2}{s}(U)$ is continuously embedded in $\C(\overline{U})$ if $s>1$ and
$U$ is a bounded Lipschitz domain, where $\overline{U}$ is the
closure of $U$ in $\IR^2$ and $\C(\overline{U})$ is the set of
all continuous functions on $\overline{U}$. Since $\Omega$ is a
bounded polygonal domain, it follows that for any $s>1$, there is a
constant $C_s>0$ such that $\| w\|_{C(\overline{\Omega})}\leq C_s\|
w\|_{\SOB{2}{s}(\Omega)}$, where $\| w\|_{C(\overline{\Omega})}=\|w\|_\infty=\max\{ |w(x, y)| \,
: \, (x, y)\in \overline{U}\}$.
%The constant $C_s$ can be computed from $s$; as remarked before, we will
%not present an algorithm for computing it in this paper; instead, we
%assume that it is computable.

\begin{lemma} \label{l:multiplication} For $s>1$, multiplication $Mul:
\SOB{2}{s}(\Omega)\times \ELL{2}(\Omega)\to \ELL{2}(\Omega)$, $(v,
w)\mapsto vw$, is $(\deltaSOB{s}\times \deltaELL{2},
\deltaELL{2})$-computable.
\end{lemma}

\begin{proof} Assume that $\{ p_k\}$ is a $\deltaSOB{s}$-name of $v$
and $\{ q_k\}$ is a $\deltaELL{2}$-name of $w$. For each $n\in\IN$,
pick $k(n)\in\IN$ such that
$C_{s}\|v\|_{\SOB{2}{s}}\|w-q_{k(n)}\|_2\leq 2^{-(n+1)}$. Since
$\|v\|_{\SOB{2}{s}}$ is computable from $\{ p_k\}$, the function
$n\mapsto k(n)$ is computable from $\{p_k\}$ and $\{ q_k\}$. Next
pick $m(n)\in\IN$ such that $\|q_{k(n)}\|_{\C(\overline{\Omega})}\|
v- p_{m(n)}\|_{\SOB{2}{s}}\leq 2^{-(n+1)}$. It is clear that $m(n)$
is computable from $k(n)$, $\{ q_k\}$, and $\{ p_k\}$. The sequence
$\{ p_{m(n)}q_{k(n)}\}_n$ is then a $\deltaELL{2}$-name of $vw$, for
it is a sequence of polynomials of rational coefficients and $\| vw
- p_{m(n)} q_{k(n)}\|_2 \leq \| v\|_{\C(\overline{\Omega})}\| w -
q_{k(n)}\|_2 + \|q_{k(n)}\|_{\C(\overline{\Omega})}\| v-
p_{m(n)}\|_{\SOB{2}{s}}\leq 2^{-n}$.
\end{proof}

%%%%%%%%%%%%%%%%%%%%%%%%%%%%%%%%%%%%%%%%%%%%%%%%%%%%%%%%%%%
\subsection{Some classical properties of fractional powers of $\IA$}
\label{ss:fp-A}

It is known that fractional powers of the Stokes operator $\IA$ are
well defined; cf. \mycite{Section 2.6}{Pazy83}. In the following, we
summarize some classical properties of the Stokes operator and its
fractional powers; these properties will be used in later proofs.

\begin{fact} \label{f:A-alpha} Let $\IA$ be the Stokes operator.
\begin{itemize}
\item[(1)] For every $0\leq \alpha\leq 1$, let
$D(\IA^{\alpha})$ be the domain of $\IA^{\alpha}$; this is a Banach
space with the norm $\|\vec
u\|_{D(\IA^{\alpha})}:=\|\IA^{\alpha}\vec
u\|_{\SOLZ{2}(\Omega)}=\|\IA^{\alpha}\vec u\|_{2}$. In particular,
$D(\IA^{\alpha})$ is continuously embedded in $\SOB{2}{2\alpha}$,
that is, for every $\vec u\in D(\IA^{\alpha})$,
\begin{equation} \label{e:H-L}
\| \vec u\|_{\SOB{2}{2\alpha}} \leq  \| \vec
u\|_{D(\IA^{\alpha})}=C\|\IA^{\alpha}\vec u\|_{2}
\end{equation}
where $C$ is a constant independent of $\alpha$. Moreover, we have
$D(\IA^{1/2})=L^{\sigma}_{2, 0}(\Omega)\bigcap \{ \vec{u}\in
(H^1(\Omega))^2; \mbox{$\vec{u}=0$ on $\partial \Omega$}\}$.
\item[(2)] For every nonnegative $\alpha$ the
estimate
\begin{equation} \label{e:alpha-bound}
\| \IA^{\alpha}e^{-t\IA}\vec u\|_{2}\leq C_{\alpha}
t^{-\alpha}\|\vec u\|_{2}, \quad t>0
\end{equation}
is valid for all $\vec u\in \SOLZ{2}(\Omega)$, where $C_{\alpha}$ is
a constant depending only on $\alpha$. In particular, $C_0=1$.
Moreover, the estimate implies implicitly that for every $\vec{u}\in
L^{\sigma}_{2, 0}(\Omega)$, $e^{-t\IA}\vec{u}$ is in the domain of
$\IA$, and thus $e^{-t\IA}\vec{u}$ vanishes on the boundary of
$\Omega$ for $t>0$.
\item[(3)] If $\alpha \geq \beta >0$, then $D(\IA^{\alpha})\subseteq
D(\IA^{\beta})$.
\item[(4)] For $0<\alpha <1$, if $\vec u\in D(\IA)$, then
\[ \IA^{\alpha}\vec u=\frac{\sin
\pi\alpha}{\pi}\int_{0}^{\infty}t^{\alpha -1}\IA(t\II+\IA)^{-1}\vec
u\,dt \]
\item[(5)] $\|\IA^{-1/4}\IP (\vec u,\nabla)\vec v\|_{2}\leq M
\|\IA^{1/4}\vec u\|_{2}\|\IA^{1/2}\vec v\|_{2}$ is valid for all
$\vec u, \vec v$ in the domain of $\IA^{3/5}$, where $M$ is a
constant independent of $\vec u$ and $\vec v$.
\end{itemize}
\end{fact}

\begin{proof} See Lemmas 2.1, 2.2 and 2.3 in \cite{Giga83a} for (1) and (2) except for
$C_0=1$; $C_0=1$ is proved in Lemma \ref{l:linear-estimate}. See
Theorems 6.8 and 6.9 in Section 2.6 of \cite{Pazy83} for (3) and
(4); Lemma 3.2 in \cite{Giga83a} for (5).
\end{proof}

We record, without going into the details, that the constants  $C$,
$M$, and $C_{\alpha}$ ($0\leq \alpha\leq 1$) appeared in
Fact~\ref{f:A-alpha} are in fact computable (some general discussions on the computability of Sobolev embedding constants and interpolation constants together with other constants in the PDE theory are forthcoming). %(from $W$ and
%$\alpha$).

%%%%%%%%%%%%%%%%%%%%%%%%%%%%%%%%%%%%%%%%%%%%%%%%%%%%%%%%%%%
\subsection{Proof of Proposition~\ref{p:convergence}} \label{ss:iteration}

In order to show that the iteration sequence is effectively
convergent, we need to establish several estimates on various
functions such as $\| \IA^{\beta}u_m(t)\|_2$ and $\|
\IA^{\beta}(u_{m+1}(t)-u_m(t))\|_2$ for $\beta$ being some positive
numbers. Subsequently, as a prerequisite, $u_m(t)$ must be in the
domain of $\IA^{\beta}$; thus the functions $u_m(t)$ are required to
have higher smoothness than the given initial function $\vec{a}$
according to Fact \ref{f:A-alpha}-(1). This is indeed the case: For
functions $u_m(t)$ obtained by the iteration (\ref{e:iteration0}),
it is known classically that if $u_m(0)\in L_2(\Omega)$ then
$u_m(t)\in H^{2\alpha}_2(\Omega)$ for $t>0$, where $0\leq \alpha
\leq 1$. In other words, $u_m(t)$ undergoes a jump in smoothness
from $t=0$ to $t>0$ (due to the integration). In the following
lemma, we  present an algorithmic version of this increase in
smoothness.

\begin{lemma} \label{H-alpha} Let $\alpha=3/5$. Then for the
iteration ~(\ref{e:iteration0})
\[ \vec u_0(t)\;=\;e^{-t\IA}\vec a
, \qquad \vec u_{m+1}(t)\;=\;\vec u_0(t)\;-\;\int^t_0e^{-(t-s)\IA
}\IB\vec u_m(s)\,ds \] the mapping $\IS_{H}: \IN\times
\SOLZ{2}(\Omega)\times (0, \infty)\to \SOB{2,
0}{2\alpha}(\Omega)\times \SOB{2, 0}{2\alpha}(\Omega)$, $(m, \vec a,
t)\mapsto \vec{u}_m(t)$,  is well-defined and $(\nu\times
\deltaSOLZ{2}\times \myrho, \deltaSOB{2\alpha})$-computable.
\end{lemma}

We emphasize that the lemma holds true for $t>0$ only. Also the
choice of $\alpha = 3/5$ is somewhat arbitrary; in fact, $\alpha$
can be selected to be any rational number strictly between
$\frac{1}{2}$ and $\frac{3}{4}$. The requirement $\alpha <
\frac{3}{4}$ guarantees that  $D(\IA^{\alpha})\subset
H^{2\alpha}_{2, 0}(\Omega)\times H^{2\alpha}_{2, 0}(\Omega)$ because
$2\alpha < 3/2$ (see (\ref{e:s-restriction})). The other condition
$\alpha > \frac{1}{2}$ ensures that Lemma \ref{l:multiplication} can
be applied for $2\alpha > 1$. \\

\begin{proof}
We induct on $m$. Note that for any $t>0$ and any $a\in
L^{\sigma}_{2, 0}(\Omega)$, the estimates (\ref{e:H-L}) and
(\ref{e:alpha-bound}) imply that
\[ \| e^{-t\IA}a\|_{H^{2\alpha}_{2}}\leq C\|
\IA^{\alpha}e^{-t\IA}a\|_{2}\leq CC_{\alpha}t^{-\alpha}\| a\|_{2}
\]
Combining this inequality with the following strengthened version of
(\ref{H^1-norm}): for any $a\in \mathcal{P}$,
\[ \sum_{n, m\geq 1}(1+n^2+m^2)^{2}|a_{n,m}|^2<\infty \]
(the inequality is valid since $a$ is $C^{\infty}$), a similar
argument used to prove Proposition \ref{p:linear-NS} works for
$m=0$. Moreover, by type conversion (Fact~\ref{f:conversion}), $a\in
L^{\sigma}_{2, 0}(\Omega) \mapsto u_0\in C((0, \infty), H^{6/5}_{2,
0}(\Omega)\times H^{6/5}_{2, 0}(\Omega))$ is
$(\delta_{L^{\sigma}_{2, 0}}, [\rho \to
\delta_{H^{6/5}_{2, 0}}])$-computable. \\

Assume that $(j, a)\mapsto u_j$ is $(\nu, \delta_{L^{\sigma}_{2,
0}}, [\rho \to \delta_{H^{6/5}_{2, 0}}])$-computable for $0\leq
j\leq m$, where $a\in L^{\sigma}_{2, 0}(\Omega)$, and $u_j\in C((0,
\infty), (H^{6/5}_{2, 0}(\Omega))^2)$. We show how to compute a
$\deltaSOB{6/5}$-name for $u_{m+1}(t)= e^{-t\IA}a
-\int^t_0e^{-(t-s)\IA }\IB u_{m}(s)ds$ on inputs $m+1$, $a$ and
$t>0$. Let us first look into the nonlinear term $\mathbb{B}u_m$. It
is clear that $\IB u_m(s)$ lies in $L^{\sigma}_{2, 0}(\Omega)$ for
$s>0$. Moreover, it follows from Lemmas~\ref{l:differentiation} and
\ref{l:multiplication}, and Proposition \ref{projection-P} that the
map $(u_m, s)\mapsto \IB u_m(s)$ is
$([\myrho\myto\deltaSOB{2\alpha}], \myrho,
\deltaSOLZ{2})$-computable for all $s\in (0, t]$. Now since $\IB
u_m(s)$ is in $L^{\sigma}_{2, 0}(\Omega)$ for $s>0$, it follows from
the case where $m=0$ that $(u_m, s)\mapsto e^{-(t-s)\IA}\IB u_m(s)$
is $([\myrho\myto\deltaSOB{2\alpha}], \myrho,
\deltaSOB{6/5})$-computable for $0 < s < t$. \\

Next let us consider the integral $\int^t_0e^{-(t-s)\IA }\IB\vec
u_m(s)\,ds$; we wish to compute a $\delta_{H^{6/5}_{2, 0}}$-name of
the integral from $a$ and $t>0$. We make use of the following fact:
For $\theta\geq 1$, the integration operator from $C([a, b];
H^{\theta}_{2, 0}(\Omega)\times H^{\theta}_{2, 0}(\Omega))$ to
$H^{\theta}_{2, 0}(\Omega)\times H^{\theta}_{2, 0}(\Omega)$,
$F\mapsto \int^{b}_{a}F(t)(x)dt$, is computable from $a$, $b$, and
$F$. This fact can be proved by a similar argument as the one used
in the proof of Lemma 3.7 \cite{WeZh05}. However, since the function
$e^{-(t-s)\IA}\mathbb{B}u_m(s)$ is not necessarily in
$(H^{6/5}_{2}(\Omega))^2$ when $s=0$ or $s=t$, the stated fact
cannot be directly applied to the given integral. To overcome the
problem of possible singularities at the two endpoints, we use a
sequence of closed subintervals $[t_n, t-t_n]$ to approximate the
open interval $(0, t)$, where $t_n=t/2^n$, $n\geq 1$.  Then it
follows from the stated fact and the induction hypotheses that a
$\delta_{H^{6/5}_{2, 0}}$-name, say $\{ p_{n, K}\}$, of
$u^n_{m+1}(t)=e^{-t\IA}a
-\int^{t-t_n}_{t_n}e^{-(t-s)\IA}\mathbb{B}u_m(s)ds$ can be computed
from inputs $n$, $u_m$, and $t>0$, which satisfies the condition
that $\| u^n_{m+1}(t) - p_{n, K}\|_{H^{6/5}_{2}}\leq 2^{-K}$. Thus
if we can show that the integrals $\int_{0}^{t_n}e^{-(t-s)\IA}\IB
u_m(s)ds$ and $\int_{t-t_n}^{t}e^{-(t-s)\IA}\IB u_m(s)ds$ tend to
zero effectively in $\SOB{2}{6/5}\times \SOB{2}{6/5}$-norm as $n\to
\infty$, then we can effectively construct a $\deltaSOB{6/5}$-name
of $u_{m+1}(t)$ from $\{ p_{n, K}\}_{n, K}$. \\

It remains to show that both sequences of integrals tend to zero
effectively in $\SOB{2}{6/5}\times \SOB{2}{6/5}$-norm as $n\to
\infty$. Since a similar argument works for both sequences, it
suffices to show that the sequence $\mbox{Int}_n :=
\int_{0}^{t_n}e^{-(t-s)\IA}\IB u_m(s)ds$ tends to zero effectively
as $n\to \infty$. We are to make use of Fact \ref{f:A-alpha}-(1),
(2), (5) for showing
the effective convergence. The following two claims comprise the proof. \\

\noindent {\bf Claim I.}  Let $\beta=\frac{1}{2}$ or $\frac{1}{4}$.
Then the map $(a, t, m, \beta) \mapsto M_{\beta, m}$ is computable,
where $M_{\beta, m}$ is a positive number satisfying the condition
\begin{equation} \label{e:beta-m-bound}
\|\IA^{\beta}u_m(s)\|_2\leq M_{\beta, m}s^{-\beta} \quad \mbox{for
all $0<s<t$}
\end{equation}
(note that $M_{\beta, m}$ is independent of $s$). \\

\noindent {\bf Proof.} Again we induct on $m$. For $m=0$, let
$M_{\beta, 0}=C_{\beta}\| a\|_2$, where $C_{\beta}$ is the constant
in estimate (\ref{e:alpha-bound}) with $\alpha$ replaced by $\beta$
and $u$ by $a$. Then $M_{\beta, 0}$ is computable from $a$ and
$\beta$, and $\| \IA ^{\beta}u_0(s)\|_{2}\leq C_{\beta}s^{-\beta}\|
a\|_2=M_{\beta, 0}s^{-\beta}$ for any $s>0$. Assume that $M_{\beta,
k}$, $0\leq k\leq m$, has been computed from $k, \beta, a$, and
$t>0$. We show how to compute $M_{\beta, m+1}$. Since $u_{m+1}(s)$
has a singularity at $s=0$, it may not be in
$H^{2\beta}_{2}(\Omega)\times H^{2\beta}_{2}(\Omega)$ at $s=0$
(recall that $D(\IA^{1/2})=L^{\sigma}_{2, 0}(\Omega)\bigcap \{
\vec{u}\in H^1_2(\Omega)\times H^1_2(\Omega): \mbox{$\vec{u}=0$ on
$\partial \Omega$}\}$).  Let us first compute a bound (in
$L_2$-norm) for $\IA^{\beta}\int^{s}_{\epsilon}e^{-(t-r)\IA}\IB
u_{m}(r)dr$, where $0<\epsilon < s$. It follows  from the induction
hypothesis, Fact~\ref{f:A-alpha}-(1), (2), (5), and Theorems 6.8 and
6.13 in \cite{Pazy83} that
\begin{align} \label{epsilon-M-beta-m}
& \| \IA^{\beta} \int^{s}_{\epsilon}e^{-(s-r)\IA}\IB
u_{m}(r)dr\|_2 \nonumber \\
& =  \| \int^{s}_{\epsilon}\IA^{\beta
+1/4}e^{-(s-r)\IA}\IA^{-1/4}\IB
u_{m}(r)dr\|_2  \nonumber  \\
& \leq  C_{\beta +1/4}\int^{s}_{\epsilon}(s-r)^{-(\beta +1/4)}\|
\IA^{-1/4}\IB
u_{m}(r)\|_2dr \nonumber \\
& \leq  C_{\beta +1/4}M\int^{s}_{\epsilon}(s-r)^{-(\beta +1/4)}\|
\IA^{1/4}
u_{m}(r)\|_2\| \IA^{1/2}u_{m}(r)\|_2dr \nonumber \\
& \leq  C_{\beta +1/4}MM_{1/4, m}M_{1/2,
m}\int^{t}_{\epsilon}(s-r)^{-(\beta +1/4)}r^{-3/4}dr
\end{align}
Subsequently, we obtain that
\begin{align} \label{M-beta-m}
& \| \IA^{\beta}u_{m+1}(s)\|_{2} \nonumber \\
& =  \| \IA^{\beta}u_0(s) - \int^{s}_{0}\IA^{\beta}e^{-(s-r)\IA}\IB
u_{m}(r)dr \|_2 \nonumber \\
& \leq M_{\beta, 0}s^{-\beta} + \| \lim_{\epsilon \to
0}\int^{s}_{\epsilon}\IA^{\beta}e^{-(s-r)\IA}\IB u_{m}(r)dr \|_2
\nonumber \\
& \leq M_{\beta, 0}s^{-\beta} + C_{\beta +
\frac{1}{4}}MM_{\frac{1}{4}, m}M_{\frac{1}{2}, m}\int
^{s}_{0}(s-r)^{-(\beta+\frac{1}{4})}r^{-3/4}dr \nonumber \\
& = M_{\beta, 0}s^{-\beta} + C_{\beta + \frac{1}{4}}MM_{\frac{1}{4},
m}M_{\frac{1}{2}, m}B(\frac{3}{4}-\beta, \frac{1}{4})s^{-\beta}
\end{align}
where $B(\frac{3}{4}-\beta, 1/4)$ is the integral
$\int_{0}^{1}(1-\theta)^{(\frac{3}{4}-\beta)-1}\theta^{\frac{1}{4}-1}d\theta$,
which is the value of the Beta function $B(x, y)=\int ^1_0
(1-\theta)^{1-x}\theta^{1-y}d\theta$ at $x=\frac{3}{4}-\beta$ and
$y=\frac{1}{4}$. It is clear that $B(\frac{3}{4}-\beta, 1/4)$ is
computable. Thus if we set
\begin{equation} \label{e:M-beta-m+1}
M_{\beta, m+1}=M_{\beta, 0} + C_{\beta +
\frac{1}{4}}MM_{\frac{1}{4}, m}M_{\frac{1}{2},
m}B\left(\frac{3}{4}-\beta, \frac{1}{4}\right)
\end{equation}
then $M_{\beta, m+1}$ is computable and satisfies the condition that
$\| \IA^{\beta}u_{m+1}(s)\|_{2}\leq M_{\beta, m+1}s^{-\beta}$ for
all $0<s<t$. The proof of Claim I is complete. \\

\noindent {\bf Claim II.} $\left\| \int^{t_n}_{0}e^{-(t-s)\IA}\IB
u_m(s)ds \right\|_{\SOB{2}{6/5}} \to 0$ effectively as $n\to \infty$. \\

\noindent {\bf Proof.} Once again, to avoid singularity of $u_m(s)$
at $s=0$, we begin with the following estimate: Let $0<\epsilon <
t_n$. Then it follows from Fact \ref{f:A-alpha}-(1), (2), (5),
(\ref{e:beta-m-bound}), (\ref{e:M-beta-m+1}), and a similar
calculation as performed in Claim I that
\begin{eqnarray*}
& & \| \int^{t_n}_{\epsilon} e^{-(t-s)\IA}\IB u_{m}(s)ds
\|_{\SOB{2}{6/5}} \\
& \leq & C\| \IA^{3/5}\int^{t_n}_{\epsilon} e^{-(t-s)\IA}\IB
u_{m}(s)ds\|_2 \\
& \leq & CC_{17/20}MM_{\frac{1}{4}, m}M_{\frac{1}{2},
m}\int^{t_n}_{\epsilon}(t-s)^{-17/20}s^{-3/4}ds \\
& \leq & CC_{17/20}MM_{\frac{1}{4}, m}M_{\frac{1}{2},
m}(t-t_n)^{-17/20}\cdot 4(t_{n}^{1/4}-\epsilon^{1/4})
\end{eqnarray*}
which then implies that
\begin{eqnarray*}
& & \| \int^{t_n}_{0} e^{-(t-s)\IA}\IB u_{m}(s)ds
\|_{\SOB{2}{6/5}} \\
& = & \| \lim_{\epsilon \to 0}\int^{t_n}_{\epsilon} e^{-(t-s)\IA}\IB
u_{m}(s)ds
\|_{\SOB{2}{6/5}} \\
& \leq & \lim_{\epsilon \to 0}CC_{17/20}MM_{\frac{1}{4},
m}M_{\frac{1}{2}, m}(t-t_n)^{-17/20}\cdot
4(t_{n}^{1/4}-\epsilon^{1/4}) \\
& = & CC_{17/20}MM_{\frac{1}{4}, m}M_{\frac{1}{2},
m}(t-t_n)^{-17/20}\cdot 4t_{n}^{1/4}
\end{eqnarray*}
It is readily seen that $\int^{t_n}_{0}e^{-(t-s)\IA}\IB
u_m(s)ds\|_{\SOB{2}{6/5}} \to 0$ effectively as $n\to \infty$
(recall that $t_n=t/2^n$). The proof for the claim II, and thus for
the lemma is now complete.
\end{proof}

\begin{remark} \label{epsilon-0} In our effort to compute an upper bound
for $\| \IA^{\beta}u_{m+1}(s)\|_2$, we start with the integral
$\int^{s}_{\epsilon}e^{-(s-r)\IA}\IB u_{m}(r)dr$ because the
integral might have a singularity at $0$; then we take the limit as
$\epsilon \to 0$ to get the desired estimate  (see computations of
(\ref{epsilon-M-beta-m}) and (\ref{M-beta-m})). The limit exists
because the bound, $C_{\beta + \frac{1}{4}}MM_{\frac{1}{4},
m}M_{\frac{1}{2}, m}B\left(\frac{3}{4}-\beta, \frac{1}{4}\right)$,
is uniform in $r$ for $0<r\leq s$. In the rest of the paper, we will
encounter several similar computations. In those later situations,
we will derive the estimates starting with $\int^t_0$ instead of
$\int^t_{\epsilon}$. There will be no loss in rigor because the
integral is uniformly bounded with respect to the integrating
variable, say $t$, for $t>0$.
\end{remark}

\begin{corollary} \label{c:H-alpha} For any $\vec a\in
L^{\sigma}_{2, 0}(\Omega)$ and $t>0$, let $\{ u_m(t)\}$ be the
sequence generated by the iteration scheme (\ref{e:iteration0})
based on $\vec a$. Then $u_m(t)\in Dom(\IA^{3/5}) \subset
Dom(\IA^{1/2}) \subset Dom(\IA^{1/4})$.
\end{corollary}

\begin{proof} The corollary follows from Lemma \ref{H-alpha} and Fact~\ref{f:A-alpha}-(3).
\end{proof}

\begin{corollary} \label{c:A-alpha}
The map from $\calP$ to $\ELL{2}(\Omega)$, $\vec{u} \mapsto
\|\IA^{\alpha}\vec{u}\|_2$, is $(\deltaSOLZ{2}, \rho)$-computable,
where $\alpha=1/8, 1/4$, or $1/2$.
\end{corollary}

\begin{proof} We prove the case when $\alpha = 1/4$; the other two
cases can be proved in exactly the same way. Since $\calP$ is
contained in the domain of $\IA$, it follows from Theorem 6.9,
Section 2.6 \cite{Pazy83} that for every $\vec{u}\in \calP$,
$\IA^{1/4}\vec{u}=\frac{\sin \pi/4}{\pi}\int^{\infty}_{0}t^{-3/4}\IA
(t\II+\IA)^{-1}\vec{u}dt$. By definition of $\calP$, if
$\vec{u}\in\calP$, then $\vec{u}$ is $\Cinfty$ with compact support
in $\Omega$, %it is a $\deltaSOLZ{2}$-name of itself,
and $\IA \vec{u}=-\IP \Laplace \vec{u} = -\Laplace \vec{u}$. Express
each component of $\vec{u}=(u^1, u^2)$ in terms of the orthogonal
basis $\{ e^{in\pi x}e^{im\pi y}\}_{n,m}$ of $\ELL{2}(\Omega)$ in
the form of $u^i=\sum_{n, m\geq 0}u^i_{n, m}e^{i\pi nx}e^{i\pi my}$,
where $u^i_{n, m}=\int^1_{-1}\int_{-1}^{1}u^1(x, y)e^{i\pi
nx}e^{i\pi my}dxdy$. Then a straightforward calculation shows that
\begin{eqnarray*}
& & \frac{\sin \pi/4}{\pi}\int^{\infty}_{0}t^{-3/4}\IA
(t\II+\IA)^{-1}u^idt \\
& = & \frac{\sin \pi/4}{\pi}\sum_{n, m\geq 0}\left(
\int_{0}^{\infty}t^{-3/4}\frac{(\pi n)^2+(\pi m)^2}{t+(\pi n)^2+(\pi
m)^2}dt\right) u^i_{n, m}e^{i\pi nx}e^{i\pi my}
\end{eqnarray*}
Since the integral is convergent and computable, it follows that
$\IA^{1/4}\vec{u}$ is computable from $\vec{u}$ and, consequently,
$\| \IA\vec{u}\|_2$ is computable.
%and then imitating the proof of
%Lemma \ref{H-alpha} it can be shown that the integral, as an
%$\ELL{2}$ function, is computable from $u$.
\end{proof}
\medskip

\begin{proof}[Proof of Proposition~\ref{p:convergence}]
For each $\vec{a}\in \SOLZ{2}$, let $\{\vec{a}_k\}$,
$\vec{a}_k=(a^1_k, a^2_k)\in \calP$, be a $\deltaSOLZ{2}$-name of
$\vec{a}$; i.e. $\| \vec{a}-\vec{a}_k\|_{2}\leq 2^{-k}$. Let
$\widetilde{C}:=c_1MB_1$, where $M$ is the constant in
Fact~\ref{f:A-alpha}(4), $c_1=\max\{C_{1/4}, C_{1/2}, C_{3/4}, 1\}$,
and $$B_1=\max\{ B(1/2, 1/4), B(1/4, 1/4), 1\}$$ with
$B(a,b)=\int_{0}^{1}(1-t)^{a-1}t^{b-1}dt$, $a, b>0$, being the beta
function. Then $M$ and $c_1$ are computable by assumption while
$B_1$ is computable for the beta functions with rational parameters
are computable. Note that $c_1B_1\geq 1$. Let
\begin{equation} \label{e:v-m}
v_m(t)=u_{m+1}(t)-u_m(t)=\int_{0}^{t}e^{-(t-s)\IA}(\IB u_m(s)-\IB
u_{m-1}(s))ds, \quad m\geq 1
\end{equation}

Our goal is to compute a constant $\epsilon$, $0<\epsilon<1$, such
that near $t=0$,
\begin{equation} \label{e:epsilon}
\| v_m(t)\|_2\leq L\epsilon^{m-1}
\end{equation}
where $L$ is a constant. Once this is accomplished, the proof is
complete.

It follows from Corollary \ref{c:H-alpha} that Fact
\ref{f:A-alpha}-(5) holds true for all $u_m(t)$ and $v_m(t)$ with
$t>0$. It is also known classically that
\begin{eqnarray} \label{v-m-bound}
&  & \| \IA^{-\frac{1}{4}}(\IB u_{m+1 }(t) - \IB u_{m}(t))\|_2 \nonumber \\
& = & \| \IA^{-\frac{1}{4}}\IB u_{m+1 }(t) - \IA^{-\frac{1}{4}}\IB u_{m}(t)\|_2 \nonumber \\
& \leq & M\left( \| \IA^{\frac{1}{4}}v_m(t)\|_2  \|
\IA^{\frac{1}{2}}u_{m+1}(t)\|_2 + \| \IA^{\frac{1}{4}}u_m(t)\|_2 \|
\IA^{\frac{1}{2}}v_m(t)\|_2 \right)
\end{eqnarray}
(see, for example,\cite{GiMi85}). The equality in the above estimate
holds true because $\IA ^{-1/4}$ is a (bounded) linear operator. The
estimate (\ref{v-m-bound}) indicates that, in order to achieve
(\ref{e:epsilon}), there is a need in establishing some bounds on
$\| \IA^{\beta}u_m(t)\|_2$ and $\| \IA^{\beta}v_m(t)\|_2$ which
become ever smaller as $m$ gets larger uniformly for values of $t$
near zero. The desired estimates are developed in a series of
claims beginning with the following one. \\

\noindent \textbf{Claim 1.} Let $\beta = \frac{1}{4}$ or
$\frac{1}{2}$; let
\[ \tilde{K}^{\vec{a}}_{\beta, 0}(T)=\max_{0\leq t\leq
T}t^{\beta}\|\IA^{\beta}e^{-t\IA}\vec{a}\|_{2} \] and
\[ k^{\vec a}_0(T)=\max\{ \tilde{K}^{\vec a}_{\frac{1}{4}, 0}(T), \tilde{K}^{\vec a}_{\frac{1}{2},
0}(T)\} \] Then there is a computable map from $L^{\sigma}_{2,
0}(\Omega)$ to $(0, 1)$, $\vec a \mapsto T_{\vec a}$, such that
\[ k^{\vec a}_0(T_{\vec a}) < \frac{1}{8\widetilde{C}} \]

\noindent \textbf{Proof.} First we note that
$t^{\beta}\|\IA^{\beta}e^{-t\IA}\vec{a}\|_{2} = 0$ for any $\vec
a\in L^{\sigma}_{2, 0}(\Omega)$ if $t=0$; cf. Theorem 6.1 in
\cite{Giga83a}. Furthermore, it follows from (\ref{C-0-I}) that the
operator norm of $e^{-t\mathbb{A}}$, $\| e^{-t\mathbb{A}}\|_{op}$,
is bound above by 1 for any $t>0$. Since $e^{-t\mathbb{A}}$ is the
identity map on $L^{\sigma}_{2, 0}(\Omega)$ when $t=0$, we conclude
that $\max_{0\leq t\leq T}\| e^{-t\mathbb{A}}\|_{op}\leq 1$ for any
$T>0$. Now for any $\vec a\in L^{\sigma}_{2, 0}(\Omega)$, it follows
from Fact ~\ref{f:A-alpha}-(2) and Theorems 6.8 and 6.13 in Section
2.6 of \cite{Pazy83} ($\IA ^{\alpha}$ and $e^{-t\IA}$ are
interchangeable) that
\begin{eqnarray*}
\tilde{K}^{\vec{a}}_{\beta, 0}(T) & = & \max_{0 \leq t \leq T} t
^\beta \|
\IA^\beta e^{- t\IA} \vec a\|_2  \\
& = & \sup_{0 \leq t \leq T} t ^\beta \| \IA^\beta e^{- t\IA} \vec
a\|_2  \\
& \leq & \sup_{0 < t \leq T} t ^\beta \| \IA^\beta e^{- t\IA} ( \vec
a - \vec a_k ) \|_2
+ \sup_{0 < t \leq T} t ^\beta \| \IA^\beta e^{- t\IA} \vec a_k\|_2  \\
& \leq & C_{\beta}\| \vec a - \vec a_k \|_2 + T^\beta \max_{0\leq
t\leq T}\| e^{-t\IA}\|_{op} \|
\IA^\beta \vec a_k \|_2  \\
& \leq & c_12^{-k }  + \max\{ T^{1/4}, T^{1/2}\} \max\{\|\IA^{1/4}
\vec a_k\|_2, \|\IA^{1/2} \vec a_k\|_2\}
\end{eqnarray*}
We note that although $\vec a$ is not necessarily in the domain of
$\IA$ but $\vec a_k\in \calP$ and $\calP$ is contained in the domain
of $\IA$; thus $\IA^{\beta}\vec a_k$ is well defined. Furthermore,
it follows from Corollary \ref{c:A-alpha} that $\|\IA^{\beta}\vec
a_k\|_2$ is computable. Clearly one can compute a positive integer
$\hat{k}$ such that
\[ 2^{-\hat{k}} < \frac{1}{16c_1\widetilde{C}} \]
then compute a positive number $T_{\vec a}$ such that
\[ \max\{ T^{1/4}_{\vec{a}}, T^{1/2}_{\vec{a}}\} \max\{\|\IA^{1/4} \vec a_{\tilde{k}}\|_2,
\|\IA^{1/2} \vec a_{\tilde{k}}\|_2\} < \frac{1}{16\widetilde{C}}
\]
The computations are performed on the inputs $\vec{a}$ and the
constants $c_1$, $M$, and $B_1$.  Consequently, $k^{\vec
a}_{0}(T_{\vec a}) < 1/(8\widetilde{C})$. The proof of Claim 1 is
complete. \\

We recall that, for a given $\vec a\in L^{\sigma}_{2, 0}(\Omega)$,
the iteration scheme (\ref{e:iteration0}) is based on the ``seed"
function $u_0(t)=e^{-t\IA}\vec a$. Claim 1 asserts that the seed
function has the property that $\max_{0\leq t\leq T_{\vec
a}}t^{\beta}\|\IA^{\beta}u_0(t)\|_2$ is bounded by $\tilde{K}^{\vec
a}_{\beta, 0}$, uniformly in $t$. We extend this property to the
iteration sequence $\{
u_m(t)\}$ in the next claim. \\

\noindent \textbf{Claim 2.} Let $\beta=\frac{1}{4}$ or
$\frac{1}{2}$. Then there is a computable map $\IN\times \SOLZ{2}
\to (0, \infty)$, $(m, \vec a)\mapsto K^{\vec a}_{\beta, m}$, such
that
\begin{equation} \label{K-beta-m}
\max_{0\leq t\leq T_{\vec a}}t^{\beta}\| \IA^{\beta}u_m(t)\|_2 \leq
K^{\vec a}_{\beta, m}
\end{equation}

\noindent \textbf{Proof.} We induct on $m$. For $m=0$, let $K^{\vec
a}_{\beta, 0}=1/(8\widetilde{C})$. Then (\ref{K-beta-m}) follows
from Claim 1. It is clear that $K^{\vec a}_{\beta, 0}$ is
computable.

For $m\geq 1$ and $t>0$, $K^{\vec a}_{\beta, m+1}$ is computed by
the recursive formula:
\begin{equation} \label{recursive-K}
K^{\vec a}_{\beta, m+1}=K^{\vec a}_{\beta, 0} + C_{\beta
+\frac{1}{4}}MB(1-\beta -\frac{1}{4}, \frac{1}{4})K^{\vec
a}_{\frac{1}{4}, m}K^{\vec a}_{\frac{1}{2}, m}
\end{equation}
The recursive formula is derived similarly as that of
(\ref{M-beta-m}).  Since the upper bound is uniformly valid for all
$0<t\leq T_{\vec{a}}$, it follows that it is also valid for
$t=0$. The proof of Claim 2 is complete. \\

In the next claim, we show that the sequences $\{ K^{\vec a}_{\beta,
m}\}$, $\beta=1/4$ or $1/2$, are bounded above with an upper bound
strictly less than $1/(2\widetilde{C})$. \\

\noindent \textbf{Claim 3.} Let $k^{\vec a}_m = \max\{ K^{\vec
a}_{\frac{1}{4}, m}, K^{\vec a}_{\frac{1}{2}, m}\}$ and let
$K=\frac{4k^{\vec a}_{0}(\sqrt{2}-1)}{\sqrt{2}}$. Then $k^{\vec
a}_{m}\leq K < \frac{1}{2\widetilde{C}}$ for all $m\geq 1$. \\

\noindent \textbf{Proof.} It follows from Claim 2 that $k^{\vec
a}_{0}=\frac{1}{8\widetilde{C}}$ and $k^{\vec a}_{m+1}\leq k^{\vec
a}_{0}+\widetilde{C}(k^{\vec a}_{m})^2$
%(from \ref{recursive-K})
(recall that $\widetilde{C}=c_1MB_1$). To get a bound on $k^{\vec
a}_{m}$, let's write $k^{\vec a}_{m}=k^{\vec a}_{0}w_{m}$. Then
$w_m$ satisfies the following inequality:
\[ k^{\vec a}_{0}w_{m+1}\leq k^{\vec a}_{0} + \widetilde{C}(k^{\vec a}_{0})^2
w^2_{m} \] which  implies that
\[ w_{m+1}\leq 1 + \widetilde{C}k^{\vec a}_{0}w^2_m = 1 + \frac{1}{8}w^2_m \]
Then a direct calculation shows that
\[ w_m \leq \frac{4(\sqrt{2}-1)}{\sqrt{2}} \]
Thus
\[ k^{\vec a}_{m}=k^{\vec a}_{0}w_{m}\leq \frac{4k^{\vec
a}_{0}(\sqrt{2}-1)}{\sqrt{2}} =
\frac{\sqrt{2}-1}{2\sqrt{2}\widetilde{C}} < \frac{1}{2\widetilde{C}}
\] And so if we pick $K=\frac{4k^{\vec
a}_{0}(\sqrt{2}-1)}{\sqrt{2}}$, then $k^{\vec a}_{m}\leq K <
\frac{1}{2\widetilde{C}}$ for all $m\geq 1$. The proof of Claim 3 is
complete. \\

Next we present an upper bound for $t^{\alpha}\|
\IA^{\alpha}v_m(t)\|_2$, $m\geq 1$. Recall that $v_m(t)=u_{m+1}(t) -
u_m(t)$. \\

\noindent \textbf{Claim 4.} For $t\in [0, T_{\vec a}]$, $0\leq
\alpha <\frac{3}{4}$, and $m\geq 1$,
\begin{equation} \label{v-m}
t^{\alpha}\|\IA^{\alpha}v_{m}(t)\|_2 \leq 2KC_{\alpha
+\frac{1}{4}}(2\tilde{C}K)^{m-1}B(1-\alpha -\frac{1}{4},
\frac{1}{4})
\end{equation}

\noindent \textbf{Proof.} First we observe that (\ref{v-m}) is true
for $t=0$. Next we assume that $0<t\leq T_{\vec a}$. Once again we
induct on $m$. At $m=1$: We recall from the definition of $c_1$ and
$B_1$ that $\frac{1}{2c_1B_1}\leq \frac{1}{2}$. Also it follows from
(\ref{v-m-bound}), Claims 2 and 3 that $\|
\IA^{\frac{1}{2}}u_1(t)\|_2 \leq K^{\vec a}_{\frac{1}{2},
1}t^{-\frac{1}{2}}\leq Kt^{-\frac{1}{2}}$, $\|
\IA^{\frac{1}{4}}u_0(t)\|_2 \leq Kt^{-\frac{1}{4}}$, $\|
\IA^{\frac{1}{4}}v_0(t)\|_2 \leq 2Kt^{-\frac{1}{4}}$, and $\|
\IA^{\frac{1}{2}}v_0(t)\|_2 \leq 2Kt^{-\frac{1}{2}}$. Making use of
these inequalities we obtain the following estimate:
\begin{align*}
  t^{\alpha} & \|  \IA^{\alpha}v_{1}(t)\|_2 \\
& =  t^{\alpha}\| \IA^{\alpha}(u_{2}(t)-u_1(t))\|_2
\\
& =  t^{\alpha}\left\| \IA^{\alpha}\int_{0}^{t}e^{-(t-s)\IA}(\IB u_1(s) - \IB u_0(s))ds\right\|_2 \\
& \leq t^{\alpha}C_{\alpha +\frac{1}{4}}\int_{0}^{t}(t-s)^{-\alpha
-\frac{1}{4}}\| \IA^{-\frac{1}{4}}\IB u_1(s) - \IA^{-\frac{1}{4}}\IB u_0(s)\|_2 ds \\
& \leq  t^{\alpha}C_{\alpha +\frac{1}{4}}\int_{0}^{t}(t-s)^{\alpha
-\frac{1}{4}}M\bigg ( \| \IA^{\frac{1}{4}}v_0(s)\|_2 \cdot \|
\IA^{\frac{1}{2}}u_{1}(s)\|_2 \\
&\qquad\qquad \qquad\qquad \qquad + \| \IA^{\frac{1}{4}}u_0(s)\|_2\cdot
\| \IA^{\frac{1}{2}}v_0(s)\|_2 \bigg ) ds \\
& \leq t^{\alpha}C_{\alpha
+\frac{1}{4}}M2K^2\int_{0}^{t}(t-s)^{-\alpha
-\frac{1}{4}}s^{-\frac{3}{4}}ds \\
& =  2KC_{\alpha +\frac{1}{4}}MKB(1-\alpha -\frac{1}{4}, \frac{1}{4}) \\
& <  2KC_{\alpha +\frac{1}{4}}\frac{M}{2c_1MB_1}B(1-\alpha
-\frac{1}{4}, \frac{1}{4}) \quad (\mbox{recall that $K<\frac{1}{2\widetilde{C}}=\frac{1}{2c_1MB_1}$}) \\
& <  2KC_{\alpha +\frac{1}{4}}(2\tilde{C}K)^0B(1-\alpha
-\frac{1}{4}, \frac{1}{4})
\end{align*}
Thus (\ref{v-m}) is true for $m=1$.

Now assuming that (\ref{v-m}) holds for all $1\leq j\leq m$, we show
that (\ref{v-m}) is also true for $m+1$. First it follows from
Claims 2 and 3, and the induction hypothesis that for any $s\in (0,
T_{\vec a})$,
\begin{align*}
\| \IA^{\frac{1}{4}} & v_{m}(s)\|_2 \cdot \|
\IA^{\frac{1}{2}}u_{m+1}(s)\|_2 \\
& \leq
2KC_{\frac{1}{4}+\frac{1}{4}}(2\widetilde{C}K)^{m-1}B(1-\frac{1}{4}-\frac{1}{4},
\frac{1}{4})s^{-\frac{1}{4}}\cdot Ks^{-\frac{1}{2}} \\
& \leq  2Kc_1(2\widetilde{C}K)^{m-1}B_1Ks^{-\frac{3}{4}}
\end{align*}
Similarly,
\[ \| \IA^{\frac{1}{2}}v_m(s)\|_2\cdot \|
\IA^{\frac{1}{4}}u_{m}(s)\|_2 \leq
2Kc_1(2\widetilde{C}K)^{m-1}B_1Ks^{-\frac{3}{4}}
\]
Thus,
\begin{align*}
& \| \IA^{\frac{1}{2}}u_{m+1}(s)\|_2\cdot \|
\IA^{\frac{1}{4}}v_{m}(s)\|_2 + \| \IA^{\frac{1}{2}}v_m(s)\|_2 \cdot
\|
\IA^{\frac{1}{4}}u_{m}(s)\|_2 \\
& \leq  2Kc_1(2\widetilde{C}K)^{m-1}B_1\cdot 2Ks^{-\frac{3}{4}}
\end{align*}
These inequalities imply the desired estimate:
\begin{align*}
 t^{\alpha}\| &  \IA^{\alpha}  v_{m+1}(t)\|_2 \\
& \leq  t^{\alpha}C_{\alpha +\frac{1}{4}}\int_{0}^{t}(t-s)^{-\alpha
-\frac{1}{4}}\| \IA^{-\frac{1}{4}}(\IB u_{m+1}(s)-\IB u_{m}(s))\|_2
ds
\\
& \leq  t^{\alpha}C_{\alpha +\frac{1}{4}}\int_{0}^{t}(t-s)^{-\alpha
-\frac{1}{4}}M \Big (\| \IA^{\frac{1}{2}}u_{m+1}(s)\|_2 \cdot \|
\IA^{\frac{1}{4}}v_{m}(s)\|_2 \\
& \qquad \qquad \qquad + \| \IA^{\frac{1}{2}}v_m(s)\|_2 \cdot \|
\IA^{\frac{1}{4}}u_{m}(s)\|_2 \Big )ds \\
& \leq  t^{\alpha}C_{\alpha +\frac{1}{4}}\int_{0}^{t}(t-s)^{-\alpha
-\frac{1}{4}}M\cdot 2Kc_{1}(2\widetilde{C}K)^{m-1}B_1\cdot
2Ks^{-\frac{3}{4}}ds
\\
& =   t^{\alpha}2KC_{\alpha +\frac{1}{4}}\cdot 2c_{1}MB_{1}K(2\widetilde{C}K)^{m-1}\int_{0}^{t}(t-s)^{-\alpha-\frac{1}{4}}s^{-\frac{3}{4}}ds  \\
& = 2KC_{\alpha +\frac{1}{4}}(2\widetilde{C}K)^{m}B(1-\alpha
-\frac{1}{4}, \frac{1}{4})
\end{align*}
The proof for Claim 4 is complete. \\

We now set $\alpha =0$,  $\epsilon = 2\widetilde{C}K$, and
$L=2KC_{\frac{1}{4}}B\left (\frac{3}{4}, \frac{1}{4}\right )$. Since
$K<\frac{1}{2\widetilde{C}}$ by Claim 3, it follows that $0<\epsilon
< 1$ and
\[ \| u_{m+1}(t)-u_m(t)\| \leq L\epsilon^{m-1} \]
Consequently, the iterated sequence $\{ u_m(t)\}$ converges
effectively to $u(t)$ and uniformly on $[0, T_{\vec a}]$.
%Since $u_{m}(0)=u(0)=\vec a$ for all
%$m\geq 0$, the sequence $\{ u_{m}(t)\}$ is actually effectively
%converges to $u(t)$ on $[0, T_{\vec a}]$.
\end{proof}

We mention in passing the following fact that can be proved by
similar computations of Claims 1 - 3: On input $(\vec{a}, m, n)$, a
positive number $T(\vec{a}, m, n)$ can be computed such that
$k^{\vec{a}}_0(T(\vec{a}, m, n)) < (8\tilde{C})^{-1}\cdot 2^{-n}$,
$T(\vec{a}, m, n+1) < T(\vec{a}, m, n)$, and $\max_{0\leq t\leq
T(\vec{a}, m, n)}t^{\beta}\| \IA^{\beta}u_m(t)\|_2\leq L^{\vec{a}}_{
\beta, m}\cdot 2^{-n}$, where $L^{\vec{a}}_{\beta, m}$ is a constant
independent of $t$ and $n$, and computable from $\vec{a}$ and $m$.

%%%%%%%%%%%%%%%%%%%%%%%%%%%%%%%%%%%%%%%%%%%%%%%%%%%%%%%%%%%
\subsection{Proof of Proposition~\ref{p:iteration0}} \label{ss:iteration0}

We now come to the proof of Proposition \ref{p:iteration0}. We need
to show that the map $\IS: \IN\times \SOLZ{2}\times [0, \infty) \to
\SOLZ{2}$, $(m, \vec{a}, t)\mapsto \vec{u}_m(t)$, is $(\nu\times
\deltaSOLZ{2}\times \myrho, \deltaSOLZ{2})$-computable. By a similar
argument as we used for proving Lemma \ref{H-alpha}, we are able to
compute $\vec{u}_m(t)$ on the input  $(m, \vec{a}, t)$, where $m\in
\mathbb{N}$, $\vec{a}\in L^{\sigma}_{2, 0}(\Omega)$, and $t>0$. We
note that $\vec{u}_m(0)=\vec{u}_0(0)=\vec{a}$ for all $m\in
\mathbb{N}$. Thus, to complete the proof, it suffices to show that
there is a modulus function $\eta: \mathbb{N}\times \mathbb{N} \to
\mathbb{N}$, computable from $\vec{a}$,  such that $\|
\vec{u}_{m+1}(t) - \vec{a}\|_2 \leq 2^{-k}$ whenever $0<t<2^{-\eta
(m+1, k)}$. Now for the details. Given $\vec{a}$ and $k$. Refereeing
to the last paragraph of the previous subsection and Fact
~\ref{f:A-alpha}-(2), (5), we obtain the following estimate: for
$0<t<T(\vec{a}, m, n)$
\begin{eqnarray*}
& & \left\| \int_{0}^{t}e^{-(t-s)\IA}\IB u_m(s)ds\right\|_2 \\
& = & \left\| \IA^{1/4}\int_{0}^{t}e^{-(t-s)\IA}\IA^{-1/4}\IB
u(s)ds\right\|_2 \\
& \leq & C_{1/4}M\int_{0}^{t}(t-s)^{-1/4}\| \IA^{1/4}u_m(s)\|_2\cdot \| \IA^{1/2}u_m(s)\|_2ds  \\
& \leq & C_{1/4}M\int^{t}_{0}(t-s)^{-1/4}\cdot s^{-1/4}\cdot
L^{\vec{a}}_{1/4, m}\cdot 2^{-n}\cdot s^{-1/2}\cdot
L^{\vec{a}}_{1/2, m}\cdot 2^{-n} ds \\
& \leq & C_{1/4}ML^{\vec{a}}_{1/4, m}L^{\vec{a}}_{1/2,
m}2^{-2n}\int^{t}_{0}(t-s)^{-1/4}s^{-3/4}ds \\
& = & C_{1/4}ML^{\vec{a}}_{1/4, m}L^{\vec{a}}_{1/2, m}B(3/4,
1/4)\cdot 2^{-2n}
\end{eqnarray*}
Thus if $\| e^{-t\IA}\vec{a} -
\vec{a}\|_2\leq 2^{-(k+1)}$ and $$2^{-2n}C_{1/4}ML^{\vec{a}}_{1/4,
m}L^{\vec{a}}_{1/2, m}B(3/4, 1/4)\leq 2^{-(k+1)}\, ,$$
then
\[ \| \vec{u}_{m+1}(t)-\vec{a}\|_2 \leq  \| e^{-t\IA}\vec{a} - \vec{a}\|_2+\left\|
\int_{0}^{t}e^{-(t-s)\IA}\IB \vec{u}_m(s)ds\right \|_{2} \leq 2^{-k}
\]
Since $e^{-t\IA}\vec{a}$ is computable in $t$ by Proposition
\ref{p:linear-NS} and $\vec{a}=e^{-0\IA}\vec{a}$, there is a
computable function $\theta_1: \IN\to \IN$ such that $\|
e^{-t\IA}\vec{a} - \vec{a}\|_2\leq 2^{-(k+1)}$ whenever $0 < t <
2^{-\theta_1(k)}$. Let $\theta_2: \IN\times \IN\to \IN$ be a
computable function satisfying $C_{1/4}ML^{\vec{a}}_{1/4,
m}L^{\vec{a}}_{1/2, m}B(3/4, 1/4)\cdot 2^{-2\theta_2(m, k)}\leq
2^{-(k+1)}$. Let $\eta (m+1, k)$ be a positive integer such that
$2^{-\eta(m+1, k)}\leq \min\{ 2^{-\theta_1(k)}, T(\vec{a}, m,
\theta_2(m, k))\}$. Then $\eta$
is the desired modulus function. The proof of Proposition \ref{p:iteration0} is complete. \\

Propositions \ref{p:convergence} and \ref{p:iteration0} show that
the solution $\vec u$ of the integral equation (\ref{e:integral-form})
is an effective limit of the computable iterated sequence $\{ {\vec
u}_m\}$ starting with ${\vec u}_0 = \vec a$ on $[0, T_{\vec a}]$;
consequently, $\vec u$ itself is also computable. Thus we obtain the
desired preliminary result:

\begin{theorem} There is a computable map
$T:\SOLZ{2}(\Omega)\to (0, \infty)$, $\vec{a}\mapsto T(\vec a)$, such that $\vec
u(t)$, the solution of the integral equation (\ref{e:integral-form}),
is computable in $\SOLZ{2}$ from $\vec a$ and $t$ for $\vec a\in \SOLZ{2}$ and
$t\in [0;T(\vec a)]$.
\end{theorem}

%%%%%%%%%%%%%%%%%%%%%%%%%%%%%%%%%%%%%%%%%%%%%%%%%%%%%%%%%%%%%%%%%
\subsection{The Inhomogeneous Case and Pressure} \label{ss:final}
It is known \mycite{Theorem~2.3}{GiMi85} that, also in the
presence of an inhomogeneity $\vec g\in\C\big([0;T],\SOLZ{2}(\Omega)\big)$,
the iterate sequence~(\ref{e:iteration}) converges to a unique
solution $\vec u$ of Equation~(\ref{e:NS-E}) near $t=0$.
Similarly to (the proofs of) Propositions~\ref{p:iteration0},
\ref{p:convergence}, and \mycite{Lemma 3.7}{WeZh05},
this solution is seen to be computable.
Moreover, $\vec g=\Helmholtz\vec f$ is computable
from $\vec f\in\big(\ELL{2}(\Omega)\big)^2$
according to Proposition~\ref{projection-P}.
Finally the right-hand side of Equation~(\ref{e:gradient}) equals
\[
\big(\II-\IP\big)[\vec f + \Laplace\vec u - (\vec u\cdot\nabla)\vec u\big]
\;=:\; \vec h
\]
which, by the definition of $\IP$ projecting onto the solenoidal
subspace, is conservative (=rotation-free/a pure divergence). Hence
the path integral $\int_{\vec 0}^{\vec x} \vec h(\vec y)\cdot
d\vec\gamma(\vec y)$ does not depend on the chosen path from $\vec
0$ to $\vec x$ and well-defines $\Pressure(\vec x)$. This concludes
our proof of Theorem~\ref{t:Main}.

%
% ---- Bibliography ----
%
% BibTeX users should specify bibliography style 'splncs04'.
% References will then be sorted and formatted in the correct style.
%
% \bibliographystyle{splncs04}
% \bibliography{mybibliography}
%

\begin{appendix}
\section{Proof of Proposition~\ref{polynomial}} \label{a:polynomial}
(a)~ For a divergence-free and boundary-free polynomial, its
coefficients must satisfy a system of linear equations. In the
following, we derive explicitly this system of linear equations in
the 2-dimensional case, i.e. $\Omega = (-1, 1)^2$. Let
$\vec\poly=(\poly_1, \poly_2)=\big(\sum_{i, j=0}^{N}a^1_{i,
j}x^iy^j, \sum_{i, j=0}^{N}a^2_{i, j}x^iy^j\big)$ be a
divergence-free and boundary-free polynomial of real coefficients.
(If the degree of $\poly_1$ or $\poly_2$ is less than $N$, then
zeros are placed for the coefficients of missing terms.) Then, by
definition,
\begin{eqnarray*}
\divergence \vec\poly & = & \frac{\partial \poly_1}{\partial x}+\frac{\partial
\poly_2}{\partial y} \\
& = & \sum_{1\leq i\leq N, 0\leq
\leq N}ia^1_{i, j}x^{i-1}y^j +
\sum_{0\leq i\leq N, 1\leq
\leq N}ja^2_{i, j}x^{i}y^{j-1} \\
& = & \sum_{0\leq i, j\leq N-1}[(i+1)a^1_{i+1, j}+(j+1)a^2_{i,
j+1}]x^{i}y^j + \\
& & + \sum_{0\leq i\leq N-1}(i+1)a^1_{i+1, N}x^iy^N + \sum_{0\leq
j\leq N-1}(j+1)a^2_{N, j+1}x^Ny^j \\
& \equiv & 0 \quad \mbox{on $\Omega$}
\end{eqnarray*}
which implies that all coefficients in $\divergence \vec\poly$ must
be zero; or equivalently,  Equation~(\ref{div}) holds true. Turning
to the boundary conditions, along the line $x=1$, since
\[ \vec\poly (1, y) = \big(\sum\nolimits_{j=0}^N(\sum\nolimits_{i=0}^N a^1_{i, j})y^j,
\sum\nolimits_{j=0}^N(\sum\nolimits_{i=0}^N a^2_{i, j})y^j \big)\]
is identically zero, it follows that $\sum_{i=0}^N a^1_{i,
j}=\sum_{i=0}^N a^2_{i, j}=0$ for $0\leq j\leq N$. There are similar
types of restrictions on the coefficients of $\vec\poly$ along the
lines $x=-1$, $y=1$, and $y=-1$. In summary, $\vec\poly$ vanishes on
$\partial\Omega$ if and only if for all $0\leq j,i\leq N$, both
(\ref{boundary-1}) and (\ref{boundary-2}) hold true.

In the 3-dimensional case, a similar calculation shows that a
polynomial triple $\vec\poly(x, y, z)=\big(\poly_1(x, y, z),
\poly_2(x, y, z), \poly_3(x, y, z)\big)$ is divergence-free and
boundary-free if and only if its coefficients satisfies a system of
linear equations with integer coefficients.

%of degree $N$, its $3(N+1)^3$ many coefficients must satisfy a
%system of $N^3+24N^2+42N+18$ many linear equations. Thus if $N\geq
%10$, the system admits infinitely many solutions; it is then
%possible to approximate the coefficients of $\poly_1$ $\poly_2$, and
%$\poly_3$ by rational numbers which also satisfy the constraints
%imposed by the system of linear equations. In other words,
%$\vec\poly$ can be approximated by polynomials in
%$\SOLPZ[\mathbb{R}^3]$. If the degrees of $\poly_1$, $\poly_2$, and
%$\poly_3$ are less than 10, we can, as discussed in the
%2-dimensional case, extend them symbolically to polynomials of
%degree 10. Therefore, every divergence-free and boundary-free
%polynomial can be approximated by polynomials in
%$\SOLPZ[\mathbb{R}^3]$. \\

(b)~ In \cite{LaVa89} it is shown that for any real number $s\geq 3$
and for any function $\vec w\in \mathcal{N}_{div}^s\cap
\SOB{2,0}{1,\sigma}(\Omega)^d$, the following holds:
\[ \inf_{\vec\poly\in \mathcal{N}^1_{\text{div}}\bigcap \calP_N^0(\Omega)^d}\|\vec w-\vec\poly\|_{\SOB{2}{s}(\Omega)^d}\leq
CN^{-2}\|\vec w\|_{\SOB{2}{s}(\Omega)^d} \]
where $\Omega=(-1,
1)^d$,
$$
\mathcal{N}_{\text{div}}^s =\{ \vec w\in
\SOB{2}{s}(\Omega)^d \, | \, \divergence  \vec{w}=0\},\quad
\calP_N^0(\Omega)=\calP_N(\Omega)\bigcap
\SOB{2,0}{1,\sigma}(\Omega),$$
$\calP_N$ is the set of all $d$-tuples
of real polynomials with $d$ variables and degree less than or equal
to $N$ with respect to each variable, $\SOB{2,0}{1,\sigma}(\Omega)$
is the closure in $\SOB{2}{1}(\Omega)$ of $\Cinfty_0(\Omega)$, and
$C$ is a constant independent of $N$. This estimate implies that
every function $\vec w\in \SOLZ{\p}$ can be approximated with
arbitrary precision by divergence-free and boundary-free real
polynomials as follows: for any $n\in\IN$, since $\{ \vec u\in
\Cinfty_0(\Omega)^d \, : \, \divergence  \vec u=0\}$ is dense in
$\SOLZ{\p}$, there is a divergence-free $\Cinfty$ function $\vec u$
with compact support in $\Omega$  such that $\|\vec w-\vec
u\|_{\ELL{\p}}\leq 2^{-(n+1)}$. Then it follows from the above
inequality that there exists a positive integer $N$ and a
divergence-free and boundary-free polynomial $\vec\poly$ of degree
$N$ with real coefficients such that $\|\vec u - \vec
\poly\|_{\ELL{\p}}\leq\|\vec u-\vec\poly\|_{H^3(\Omega)^d}\leq
2^{-(n+1)}$. Consequently,
$\|\vec w-\vec \poly\|_{\ELL{\p}}\leq\|\vec w-\vec u\|_{\ELL{\p}}+\|\vec u-\vec \poly\|_{\ELL{\p}}\leq 2^{-n}$. \\

It remains to show that $\SOLPZ[\mathbb{R}^2]$, the divergence-free
and boundary-free  polynomial tuples with \emph{rational}
coefficients, is dense (in $\ELL{\p}$-norm) in the set of all
polynomial tuples with \emph{real} coefficients which are
divergence-free on $\Omega$ and boundary-free on $\partial \Omega$.
To this end we note that, according to part (a), the divergence-free
and boundary-free polynomials can be characterized, independent of
their coefficient field, in terms of a homogeneous system of linear
equations with integer coefficients. Then it follows from the lemma
below that the set of the rational solutions of this system is dense
in the set of its real solutions. And since $\Omega$ is bounded
(=relatively compact), the approximations to its coefficients of a
polynomial yields (actually uniform) the approximations to the
polynomial itself:
\[ \sup_{\vec x\in\Omega} |\poly_k(\vec x)| \;\leq\; \sum_{i, j=0}^{N} |a^k_{i,j}|\cdot M^{i+j}
\quad\text{ for } \Omega\subseteq[-M,+M]^2 \quad\text{and } k=1,2
\]
\begin{lemma}
Let $A\in\IQ^{m\times n}$ be a rational matrix.
Then the set $\operatorname{kernel}_{IQ}(A):=\{\vec x\in\IQ^n:A\cdot\vec x=\vec 0\}$ of
rational solutions to the homogeneous system of linear equations given by $A$
is dense in the set $\operatorname{kernel}_{\IR}(A)$ of real solutions.
\end{lemma}
\begin{proof}
For $d:=\operatorname{dim}\big(\operatorname{kernel}_{\IR}(A)\big)$,
Gaussian Elimination yields a basis $B=(\vec b^1,\ldots,\vec b^d)$
of $\operatorname{kernel}(A)$; in fact it holds $B\in\IQ^{n\times d}$ and
\[ \operatorname{kernel}_{\IF}(A) \;=\; \operatorname{image}_{\IF}(B)
 \;:=\; \big\{\lambda_1\vec b^1+\cdots+\lambda_d\vec b^d:\lambda_1,\ldots,\lambda_d\in\IF\big\}  \]
for \emph{every} field $\IF\supseteq\IQ$: Observe that the
elementary row operations Gaussian Elimination employs to transform
$A$ into echelon form containing said basis $B$ consist only of
arithmetic (=field) operations! (We deliberately do not require $B$
to be orthonormal; cf. \cite[\S3]{ZiBr04}.) Now
$\operatorname{image}_{\IQ}(B)$ is obviously dense in
$\operatorname{image}_{\IR}(B)$.
\end{proof}

%%%%%%%%%%%%%%%%%%%%%%%%%%%%%%%%%%%%%%%%%%%%%%%%%%%%%%%%%%%%%%%%%%%%%%
\section{Proof of Lemma~\ref{l:convolution}} \label{a:convolution}
Note that $\gamma_n\ast\Trim_k\vec\poly = (\gamma_n\ast
\Trim_k\poly_1, \gamma_n\ast \Trim_k\poly_2)$.  For each
$\vec\poly\in\SOLPZ[\mathbb{R}^2]$ and $n\geq k$,  since
\begin{align} \label{gamma_k}
\frac{\partial (\gamma_n\ast\Trim_k\poly_1)}{\partial x}& (x, y)  =
\frac{\partial}{\partial x}\int^{1}_{-1}\int^{1}_{-1}\gamma_n(x-s, y-t)\cdot\Trim_k\poly_1(s, t)\,ds\,dt \nonumber \\
& =  \int^{1-2^{-k}}_{-1+2^{-k}}\left[ \int^{1-2^{-k}}_{-1+2^{-k}}\frac{\partial \gamma_n}{\partial x}(x-s, y-t)\cdot
\Trim_k\poly_1(s, t)\,ds\right] \,dt \nonumber \\
& =  \int^{1-2^{-k}}_{-1+2^{-k}}\left[
\int^{1-2^{-k}}_{-1+2^{-k}}-\frac{\partial \gamma_n}{\partial
s}(x-s, y-t)\cdot\Trim_k\poly_1(s, t)\,ds\right] \,dt
%& = & \int^{1}_{-1}\left[ \int^{-1+2^{-k}}_{-1}-\frac{\partial \gamma_n}{\partial s}(x-s, y-t)\,\poly_1(s, t)\,ds + \right. \nonumber \\
%& &        \left. + \int^{1-2^{-k}}_{-1+2^{-k}}-\frac{\partial
%\gamma_n}{\partial s}(x-s, y-t)\,\poly_1(s, t)\,ds +
%\int^{1}_{1-2^{-k}}-\frac{\partial \gamma_n}{\partial s}(x-s, y-t)\,\poly_1(s, t)\,ds\right] \,dt \nonumber \\
%& = & \int^{1}_{-1}\left[ \int^{1-2^{-k}}_{-1+2^{-k}}-\frac{\partial
%\gamma_n}{\partial s}(x-s, y-t)\,\poly_1(s, t)\,ds \right] \,dt \nonumber \\
\end{align}
for $\Trim_k\poly_1=0$ in the exterior region of $\Omega_k$
including its boundary $\partial\Omega_k$. Note that $\frac{\partial
\gamma_n}{\partial s}$ is continuous on $\IR^2$;
$\frac{\partial\gamma_n}{\partial s}(x-s, y-t)\cdot\Trim_k\poly_1(s,
t)$ is continuous on $[-1, 1]^2$ for any given $x, y\in \IR$;
$\frac{\partial \Trim_k\poly_1}{\partial s}(s, t)$ is continuous in
$(-1+2^{-n}, 1-2^{-n})$  and $\Trim_k\poly_1$ is continuous on
$[-1+2^{-n}, 1-2^{-n}]$ for any given $t\in [-1; 1]$. Thus, we can
apply the integration by parts formula to the integral
$$\int^{1-2^{-k}}_{-1+2^{-k}}-\frac{\partial \gamma_n}{\partial
s}(x-s, y-t)\cdot\Trim_k\poly_1(s, t)\,ds$$ as follows:
\begin{align} \label{(II)}
& \quad \int^{1-2^{-k}}_{-1+2^{-k}} -\frac{\partial
\gamma_n}{\partial
s}(x-s, y-t)\cdot\Trim_k\poly_1(s, t)\,ds  \nonumber \\
&\quad = -\gamma_n(x-s, y-t)\cdot\Trim_k\poly_1(s, t)\big|
^{1-2^{-k}}_{-1+2^{-k}} \nonumber \\
&\qquad\qquad \qquad +\int^{1-2^{-k}}_{-1+2^{-k}}\gamma_n(x-s, y-t)
\cdot\frac{\partial \Trim_k\poly_1}{\partial s}(s, t)\,ds \nonumber \\
& \quad =  \int^{1-2^{-k}}_{-1+2^{-k}}\gamma_n(x-s, y-t)\cdot\frac{\partial
\Trim_k\poly_1}{\partial s}(s, t)\,ds
\end{align}
Then it follows from (\ref{gamma_k}) and (\ref{(II)})  that for any
$(x, y)\in \Omega$, %the
%integral $\int^{1}_{-1}\int^{1}_{-1}\gamma_n(x-s, y-t)\frac{\partial
%\poly_1}{\partial s}(s, t)\,ds\,dt$ exists and
\[ \frac{\partial \gamma_n\ast \Trim_k\poly_1}{\partial x}(x, y)=
\int^{1-2^{-k}}_{-1+2^{-k}}\int^{1-2^{-k}}_{-1+2^{-k}}\gamma_n(x-s,
y-t)\cdot\frac{\partial \Trim_k\poly_1}{\partial s}(s, t)\,ds\,dt \enspace . \]
A similar calculation yields that for any $(x, y)\in\Omega$, %the integral
%$\int^{1}_{-1}\int^{1}_{-1}\gamma_n(x-s, y-t)\frac{\partial
%\poly_2}{\partial t}(s, t)\,ds\,dt$ exists and
\[ \frac{\partial \gamma_n\ast \Trim_k\poly_2}{\partial y}(x, y)=
\int^{1-2^{-k}}_{-1+2^{-k}}\int^{1-2^{-k}}_{-1+2^{-k}}\gamma_n(x-s,
y-t)\cdot\frac{\partial \Trim_k\poly_2}{\partial t}(s, t)\,ds\,dt
\enspace . \]
Thus, for any $(x, y)\in\Omega$ and $n\geq k$,
\begin{align*}
\divergence & (\gamma_n\ast \Trim_k\vec\poly)(x, y)  =  \frac{\partial \gamma_n\ast
\Trim_k\poly_1}{\partial
x}(x, y)+ \frac{\partial \gamma_n\ast \Trim_k\poly_2}{\partial y}(x, y) \\
& =
\int^{1-2^{-k}}_{-1+2^{-k}}\int^{1-2^{-k}}_{-1+2^{-k}}\gamma_n(x-s,
y-t)\cdot\left[ \frac{\partial \Trim_k\poly_1}{\partial
s}+\frac{\partial \Trim_k\poly_2}{\partial t}\right] (s, t) \,ds\,dt
\\
& =0
\end{align*}
for $\Trim_k\vec\poly=(\Trim_k\poly_1, \Trim_k\poly_2)$ is
divergence-free on $\Omega_k$. This proves that for any
$\vec\poly\in\SOLPZ[\mathbb{R}^2]$ and $n\geq k$, $\gamma_n\ast
\Trim_k\vec\poly$ is divergence-free on $\Omega$.%\qed

%%%%%%%%%%%%%%%%%%%%%%%%%%%%%%%%%%%%%%%%%%%%%%%%%%%%%%%%%%%%%%%%%%%%%%
\section{Proof of Lemma \ref{l:dense}} \label{a:dense}

Since for each $\vec\poly\in\SOLPZ[\mathbb{R}^2]$ and $k\in\IN$,
$\gamma_n\ast\Trim_k\vec\poly\to \Trim_k\vec\poly$ effectively and
uniformly on $\Omega_k$ as $n\to \infty$, it suffices to show that
$\{\Trim_k\vec\poly:k\in\IN, \vec\poly\in\SOLPZ[\mathbb{R}^2]\}$ is
dense in $\SOLZ{\p}(\Omega)$. On the other hand, since
$\SOLPZ[\mathbb{R}^2]$ is dense in $\SOLZ{\p}(\Omega)$, we only need
to show that for each $\vec\poly\in \SOLPZ[\mathbb{R}^2]$ and
$m\in\IN$, there is a $k\in\IN$ such that
$2^{-m}\geq\|\vec\poly-\Trim_k\vec\poly\|_\infty=
\max\{|\poly_1(\vec x)-\Trim_k\poly_1(\vec x)|,|\poly_2(\vec
x)-\Trim_k\poly_2(\vec x)|:\vec x\in\bar\Omega\}$.

Since $\poly_i$ is uniformly continuous on $\bar\Omega$, there
exists a $k\in\IN$ such that
$|\poly_i(x,y)-\poly_i(x',y')|\leq2^{-m}$ whenever
$|x-x'|,|y-y'|\leq2^{-k+1}$, and, in particular, for
$x'=\frac{x}{1-2^{-k}}$ and $y'=\frac{y}{1-2^{-k}}$. Also, since
$\poly_i(x,y)=0$ for $(x,y)\in\partial\Omega$,
$|\poly_i(x,y)|\leq2^{-m-1}$ for all
$(x,y)\in\Omega\setminus\Omega_k$. This establishes
$|\poly_i(x,y)-\Trim_k\poly_i(x,y)|\leq2^{-m}$ on $\bar\Omega$. %\qed

%%%%%%%%%%%%%%%%%%%%%%%%%%%%%%%%%%%%%%%%%%%%%%%%%%%%%%%%%%%%%%%%%%%%%%
\end{appendix}

\end{document}